\documentclass[12pt]{article}

\usepackage{fullpage}
\usepackage[utf8]{inputenc}
\usepackage{amssymb}
\usepackage{amsmath}
\usepackage{mathtools}
\usepackage{amsfonts}
\usepackage{amsthm}
\usepackage{bbm}
\usepackage{bm}
\usepackage{hyperref}
\usepackage{physics}

\newcommand{\transpose}[1]{\ensuremath{\prescript{t}{}{#1}}}
\newcommand{\e}[0]{\ensuremath{\mathrm{e}}}

\newtheorem{theorem}{Theorem}[section]
\newtheorem{corollary}[theorem]{Corollary}
\newtheorem{lemma}[theorem]{Lemma}
\newtheorem{proposition}[theorem]{Proposition}

\theoremstyle{definition}

\numberwithin{equation}{section}

\begin{document}
\title{Bounds for theta sums in higher rank II\footnote{Research supported by EPSRC grant EP/S024948/1}}
\date{11 May 2023}
\author{Jens Marklof and Matthew Welsh}
\maketitle


\begin{abstract}
In the first paper of this series we established new upper bounds for multi-variable exponential sums associated with a quadratic form. The present study shows that if one adds a linear term in the exponent, the estimates can be further improved for almost all parameter values. Our results extend the bound for one-variable theta sums obtained by Fedotov and Klopp in 2012.
\end{abstract}


\section{Introduction}
\label{sec:intro}

For $M > 0$, a real $n\times n$ symmetric matrix $X$, and $\bm{x}, \bm{y} \in \mathbb{R}^n$, we define a {\em theta sum} as the exponential sum
\begin{equation}
  \label{eq:SMXdef}
  \theta_f(M, X, \bm{x}, \bm{y}) = \sum_{\bm{m} \in \mathbb{Z}^n} f\left( M^{-1} (\bm{m}+ \bm{x}) \right) \e\left( \tfrac{1}{2} \bm{m} X \transpose{\bm{m}} + \bm{m} \transpose{\bm{y}} \right),
\end{equation}
where $f : \mathbb{R}^n \to \mathbb{C}$ is a rapidly decaying cut-off and $\e ( z) = \e^{2\pi \mathrm{i} z}$ for any complex $z$.
If $f=\chi_\mathcal{B}$ is the characteristic function of a bounded set $\mathcal{B}\subset \mathbb{R}^n$ we have the finite sum
\begin{equation}
  \label{eq:SMXxydef}
  \theta_f( M, X, \bm{x}, \bm{y}) = \sum_{\bm{m} \in \mathbb{Z}^n \cap (M \mathcal{B} - \bm{x})} \e( \tfrac{1}{2} \bm{m} X \transpose{\bm{m}} + \bm{m}\transpose{\bm{y}}).   
\end{equation}
In this case we will also use the notation $\theta_f=\theta_{\mathcal{B}}$. In this paper we will focus on the case when $\mathcal{B}$ is the open rectangular box $(0,b_1)\times\cdots\times (0,b_n)\subset\mathbb{R}^n$. The theorems below remain valid if $f=\chi_\mathcal{B}$ is replaced by any function $f$ in the Schwartz class $\mathcal{S}(\mathbb{R}^n)$ (infinitely differentiable, with rapid decay of all derivatives). The results in the latter case follow from a simpler version of the argument for the sharp truncation, so we do not discuss them here.

 The principal result of part I \cite{MarklofWelsh2021a} in this series is the following. 

\begin{theorem}
  \label{theorem:thetasumbound1}
  Fix a compact subset $\mathcal{K}\subset\mathbb{R}_{>0}^n$, and let $\psi : [0,\infty) \to [1,\infty)$ be an increasing function such that
  \begin{equation}
    \label{eq:psisum1}
    \int_0^\infty \psi(t)^{-2n -2} dt < \infty.
  \end{equation}
  Then there exists a subset $\mathcal{X}(\psi) \subset \mathbb{R}^{n\times n}_{\mathrm{sym}}$ of full Lebesgue measure such that
  \begin{equation}
    \label{eq:thetasumbound1}
    \theta_\mathcal{B}(M, X, \bm{x}, \bm{y})  = O_X \big( M^{\frac{n}{2}} \psi(\log M) \big)
  \end{equation}
  for all $M\geq 1$, $\bm{b}=(b_1,\ldots,b_n)\in\mathcal{K}$, $X \in \mathcal{X}(\psi)$, $\bm{x},\bm{y}\in \mathbb{R}^n$. 
  The implied constants are independent of $M$, $\bm{b}$, $\bm{x}$ and $\bm{y}$.
\end{theorem}

For example, for any $\epsilon > 0$, the function $\psi(x) = (x + 1)^{\frac{1}{2n + 2} + \epsilon}$ satisfies the condition (\ref{eq:psisum1}), which produces the bound $M^{\frac{n}{2}} ( \log M)^{\frac{1}{2n + 2} + \epsilon}$ for almost every $X$ and any $\bm{x}$ and $\bm{y}$.
This improved the previously best bound due to Cosentino and Flaminio \cite{ConsentinoFlaminio2015} by a factor of $(\log M)^n$.
Moreover, in the case $n =1$, theorem \ref{theorem:thetasumbound1} recovers the optimal result obtained by Fiedler, Jurkat and K\"orner \cite{FiedlerJurkatKorner1977}.

In what follows we establish a stronger bound than (\ref{eq:thetasumbound1}), for example $M^{\frac{n}{2}} (\log M)^{\frac{1}{2n + 4} + \epsilon}$, but now only valid for almost every $\bm{y}$.
In the case $n =1$, theorem \ref{theorem:thetasumbound2} recovers theorem 0.1 of Fedotov and Klopp \cite{FedotovKlopp2012}. 

\begin{theorem}
  \label{theorem:thetasumbound2}
  Fix a compact subset $\mathcal{K}\subset\mathbb{R}_{>0}^n\times\mathbb{R}^n$, and let $\psi : [0,\infty) \to [1,\infty)$ be an increasing function such that
  \begin{equation}
    \label{eq:psisum2}
    \int_0^\infty \psi(t)^{-2n -4} dt < \infty.
  \end{equation}
  Then there exists a subset $\tilde{\mathcal{X}}(\psi) \subset \mathbb{R}^{n\times n}_{\mathrm{sym}} \times \mathbb{R}^n$ of full Lebesgue measure such that
  \begin{equation}
    \label{eq:thetasumbound2}
    \theta_\mathcal{B}(M, X, \bm{x}, \bm{y})  = O_{X,\bm{y}} \big( M^{\frac{n}{2}} \psi(\log M) \big)
  \end{equation}
  for all $M\geq 1$, $(\bm{b},\bm{x})\in\mathcal{K}$, and $(X, \bm{y}) \in \tilde{\mathcal{X}}(\psi)$.
  The implied constants are independent of $M$, $\bm{b}$ and $\bm{x}$.
\end{theorem}

The paper is organized as follows.
In section \ref{sec:theta} we review some basic properties of theta functions and the Jacobi group.
The Jacobi group is defined as the semi-direct product $H \rtimes G$ of the Heisenberg group $H$ and the symplectic group $G = \mathrm{Sp}(n, \mathbb{R})$, and, following a construction due to Lion and Vergne \cite{LionVergne1980}, the theta function associated to a Schwartz function $f\in\mathcal{S}(\mathbb{R}^n)$ is a function $\Theta_f : H \rtimes G \to \mathbb{C}$ that, for appropriate $g \in G$ and $h\in H$, is a simple rescaling of the theta sums $\theta_f$.
The theta functions $\Theta_f$ satisfy an automorphy equation, theorem \ref{theorem:thetaautomorphy}, under a certain subgroup $\tilde{\Gamma} \subset H \rtimes G$.
This subgroup, defined in section \ref{sec:Gamma}, projects to the discrete subgroup $\Gamma = \mathrm{Sp}(n, \mathbb{Z}) \subset G$.

In order to exploit additional savings from the linear term parameterized by $\bm{y}$, we found it necessary to have a better understanding of the shape of the cusp of $\Gamma \backslash G$ than in the first paper in this series \cite{MarklofWelsh2021a}. For this reason we define in section \ref{sec:fundamentaldomain} a new fundamental domain for $\Gamma \backslash G$ which has ``box-shape'' cusps, as explicated in section \ref{sec:cuspshape}.

Section \ref{sec:mainproof} contains the proof of theorem \ref{theorem:thetasumbound2}, which is based on a Borel-Cantelli type argument together with a multi-dimensional dyadic decomposition of the characteristic function of the open unit cube $(0,1)^n$ that is naturally realized as an action of the diagonal subgroup of $G$.
The execution of the Borel-Cantelli argument rests on a kind of ``uniform continuity'' property of a certain height function on $H \rtimes G$ that controls the theta function $\Theta_f$, see corollary \ref{corollary:upperbound}.
The required property is proved in section \ref{sec:heightsvolumes}, see lemma \ref{lemma:heightcontinuity}, whose proof is the motivation for the creation of the fundamental domain and the study of its cuspidal regions in sections \ref{sec:fundamentaldomain} and \ref{sec:cuspshape}.
We remark that the interaction of the dyadic decomposition with the $H$ coordinate in the Jacobi group leads to additional complications not seen in \cite{MarklofWelsh2021a}, see section \ref{sec:mainproof1}.

\section{Theta functions and the Jacobi group}
\label{sec:theta}

The theta function $\Theta_f$ associated to a Schwartz function $f\in \mathcal{S}(\mathbb{R}^n)$ is a complex-valued function defined on the Jacobi group $H \rtimes G$, the semi-direct product of the Heisenberg group $H$ with the rank $n$ symplectic group $G = \mathrm{Sp}(n, \mathbb{R})$.
Here $H$ is the set $\mathbb{R}^n \times \mathbb{R}^n \times \mathbb{R}$ with multiplication given by
\begin{equation}
  \label{eq:Hmultiplication}
  (\bm{x}_1, \bm{y}_1, t_1)(\bm{x}_2, \bm{y}_2, t_2) = (\bm{x}_1 + \bm{x}_2, \bm{y}_1 + \bm{y}_2, t_1 + t_2 + \tfrac{1}{2}( \bm{y}_1 \transpose{\bm{x}_2} - \bm{x}_1 \transpose{\bm{y}_2})),
\end{equation}
and $G$ is the group of $2n \times 2n$ real matrices $g$ preserving the standard symplectic form:
\begin{equation}
  \label{eq:Gdef}
  g
  \begin{pmatrix}
    0 & -I \\
    I & 0
  \end{pmatrix}
  \transpose{g} =
  \begin{pmatrix}
    0 & -I \\
    I & 0
  \end{pmatrix} 
\end{equation}
with $I$ the $n\times n$ identity.
Alternatively, writing $g$ in $n\times n$ blocks,
\begin{equation}
  \label{eq:Gdef1}
  G =
  \left\{
    \begin{pmatrix}
      A & B \\
      C & D
    \end{pmatrix}
    : A\transpose{B} = B\transpose{A},\ C\transpose{D} = D \transpose{C},\ A\transpose{D} - B\transpose{C} = I \right\}. 
\end{equation}
We note that $G$ acts on $H$ by automorphisms via
\begin{equation}
  \label{eq:GHaction}
  h^g = (\bm{x}A + \bm{y}C, \bm{x}B + \bm{y}D, t),\ \mathrm{where\ }h = (\bm{x}, \bm{y}, t),\ g =
  \begin{pmatrix}
    A & B \\
    C & D
  \end{pmatrix}
  ,
\end{equation}
so we may define the semidirect product $H \rtimes G$, the Jacobi group, with multiplication
\begin{equation}
  \label{eq:HGmultiplication}
  (h_1, g_1)(h_2, g_2) = (h_1 h_2^{g_1^{-1}}, g_1 g_2). 
\end{equation}

The theta function is defined by
\begin{equation}
  \label{eq:Thetadef}
  \Theta_f(h, g) = \sum_{\bm{m} \in \mathbb{Z}^n} (W(h)R(g)f)(\bm{m}),
\end{equation}
where $W$ is the Schr\"odinger representation of $H$ and $R$ is the Segal-Shale-Weil (projective) representation of $G$.
We refer the reader to \cite{MarklofWelsh2021a} for details regarding these representations, including the slightly non-standard definition of $W$ and the unitary cocycle $\rho: G \times G \to \mathbb{C}$ satisfying $R(g_1g_2) = \rho(g_1, g_2) R(g_1)R(g_2)$.
We recall here that for
\begin{equation}
  \label{eq:gXY}
  g =
  \begin{pmatrix}
    I & X \\
    0 & I
  \end{pmatrix}
  \begin{pmatrix}
    Y^{\frac{1}{2}} & 0 \\
    0 & \transpose{Y}^{-\frac{1}{2}}
  \end{pmatrix}
  \in G,
\end{equation}
we have
\begin{multline}
  \label{eq:Thetaexample}
  \Theta_f((\bm{x}, \bm{y}, t), g) \\
  = (\det Y)^{\frac{1}{4}} \e( -t + \tfrac{1}{2} \bm{x}\transpose{\bm{y}}) \sum_{\bm{m} \in \mathbb{Z}^n} f( (\bm{m} + \bm{x}) Y^{\frac{1}{2}} ) \e( \tfrac{1}{2}( \bm{m} + \bm{x}) X \transpose{(\bm{m} + \bm{x})} + \bm{m}\transpose{\bm{y}}). 
\end{multline}
For $f(\bm{x}) = \exp( -\pi \bm{x}\transpose{\bm{x}})$ and $h = (0,0,0)$, we recover $(\det Y)^{\frac{1}{4}}$ times the classical Siegel theta series that is holomorphic in the complex symmetric matrix $Z = X + \mathrm{i}Y$.
Here we choose $Y^{\frac{1}{2}}$ to be the upper-triangular matrix with positive diagonal entries such that $Y^{\frac{1}{2}} \transpose{Y}^{\frac{1}{2}} = Y$, and we emphasize that $Y^{-\frac{1}{2}}$ is always interpreted as $(Y^{\frac{1}{2}})^{-1}$ and not $(Y^{-1})^{\frac{1}{2}}$.

For general $g\in G$ we have the Iwasawa decomposition,
\begin{equation}
  \label{eq:Iwasawa}
  g =
  \begin{pmatrix}
    A & B \\
    C & D
  \end{pmatrix}
  =
  \begin{pmatrix}
    I & X \\
    0 & I
  \end{pmatrix}
  \begin{pmatrix}
    Y^{\frac{1}{2}} & 0 \\
    0 & \transpose{Y}^{-\frac{1}{2}}
  \end{pmatrix}
  \begin{pmatrix}
    \mathrm{Re}(Q) & -\mathrm{Im}(Q) \\
    \mathrm{Im}(Q) & \mathrm{Re}(Q) 
  \end{pmatrix}
  ,
\end{equation}
where $X, Y$ are symmetric and $Q$ is unitary.
Explicitly, we have
\begin{align}
  \label{eq:XYQ}
  Y & = ( C\transpose{C} + D \transpose{D})^{-1} \nonumber \\
  X & = (A \transpose{C} + B \transpose{D})( C\transpose{C} + D \transpose{D})^{-1} \nonumber \\
  Q & = \transpose{Y}^{\frac{1}{2}} ( D + \mathrm{i} C).
\end{align}
We often further decompose $Y = U V \transpose{U}$ with $U$ upper-triangular unipotent and $V$ positive diagonal, so $Y^{\frac{1}{2}} = U V^{\frac{1}{2}}$.
It is easy to express the Haar measure $\mu$ on $G$ in these coordinates,
\begin{equation}
  \label{eq:Haarmeasure}
  \dd \mu(g) = \dd Q \prod_{1\leq i\leq j\leq n} \dd x_{ij} \prod_{1\leq i < j \leq n} \dd u_{ij} \prod_{1\leq j \leq n} v_j^{-n + j -2} \dd v_{jj},
\end{equation}
where $\dd Q$ is Haar measure on $\mathrm{U}(n)$ and $\dd x_{ij}$, $\dd u_{ij}$, $\dd v_{jj}$ are respectively the Lebesgue measures on the entries of $X$, $U$, $V$. 
We can also express the Haar measure on the open, dense set of $g$ which can be written as
\begin{equation}
  \label{eq:LUdecomp}
  g =
  \begin{pmatrix}
    I & X \\
    0 & I 
  \end{pmatrix}
  \begin{pmatrix}
    A & 0 \\
    0 & \transpose{A}^{-1}
  \end{pmatrix}
  \begin{pmatrix}
    I & 0 \\
    T & I
  \end{pmatrix}
\end{equation}
with $A \in \mathrm{GL}(n, \mathbb{R})$ and $X$ and $T$ symmetric.
In these coordinates we have
\begin{equation}
  \label{eq:Haarmeasure1}
  \dd \mu(g) = c (\det A)^{-2n -1} \prod_{1\leq i\leq j\leq n} \dd x_{ij} \prod_{1\leq i, j\leq n} \dd a_{ij} \prod_{1\leq i \leq j \leq n} \dd t_{ij}
\end{equation}
where $c$ is a positive constant and $\dd x_{ij}$, $\dd a_{ij}$, $\dd t_{ij}$ are respectively the Lebesgue measure on the entries of $X$, $A$, $T$, see \cite{MarklofWelsh2021a}.
We note that the Haar measure $\tilde{\mu}$ on the Jacobi group is simply
\begin{equation}
  \label{eq:Haarmeasure2}
  \dd \tilde{\mu}(h, g) = \dd {\bm{x}}\, \dd \bm{y}\, \dd t \, \dd \mu(g),
\end{equation}
with $h = (\bm{x}, \bm{y}, t)$ and $\dd \bm{x}$, $\dd \bm{y}$, and $\dd t$ the Lebesgue measures. 

We often make use of the following refinements of the Iwasawa decomposition.
For $1 \leq l \leq n$ and the same $Q$ as in (\ref{eq:Iwasawa}), we write $g \in G$ as
\begin{multline}
  \label{eq:Pjdecomp}
  \begin{pmatrix}
    I & R_l & T_l - S_l \transpose{R}_l & S_l \\
    0 & I & \transpose{S}_l & 0 \\
    0 & 0 & I & 0 \\
    0 & 0 & -\transpose{R}_l & I
  \end{pmatrix}
  \begin{pmatrix}
    U_lV_l^{\frac{1}{2}} & 0 & 0 & 0 \\
    0 & Y_l^{\frac{1}{2}} & 0 & X_l \transpose{Y}_2^{-\frac{1}{2}} \\
    0 & 0 & \transpose{U}_l^{-1}V_l^{-\frac{1}{2}} & 0 \\
    0 & 0 & 0 & \transpose{Y_l}^{-\frac{1}{2}}
  \end{pmatrix}
  \begin{pmatrix}
    \mathrm{Re}(Q) & - \mathrm{Im}(Q) \\
    \mathrm{Im}(Q) & \mathrm{Re}(Q)
  \end{pmatrix}
  ,
\end{multline}
where $R_l$ and $S_l$ are $l \times (n-l)$ matrices, $T_l$ is $l \times l$ symmetric, $U_l$ is $l\times l$ upper-triangular unipotent, $V_l$ is $l\times l$ positive diagonal, $X_l$ is $(n-l) \times (n-l)$ symmetric, and $Y_l$ is $(n-l)\times (n-l)$ positive definite symmetric. 
We note that for $l =n$ we recover $X = T_l$ and the factorization $Y = U_l V_l \transpose{U}_l$.
In what follows we use $g_l = g_l(g) \in \mathrm{Sp}(n - l, \mathbb{R})$ to denote the matrix
\begin{equation}
  \label{eq:gldef}
  g_l =
  \begin{pmatrix}
    I & X_l \\
    0 & I 
  \end{pmatrix}
  \begin{pmatrix}
    Y_l^{\frac{1}{2}} & 0 \\
    0 & \transpose{Y}_l^{-\frac{1}{2}}
  \end{pmatrix}
  .
\end{equation}

These decompositions are closely related to the Langlands decompositions of the maximal parabolic subgroups $P_l$ of $G$.
For $1\leq l < n$, $P_l$ is the subgroup of $g \in G$ which can be written in the form
\begin{equation}
  \label{eq:Pldef}
    \begin{pmatrix}
    I & R_l & T_l - S_l \transpose{\!R}_l & S_l \\
    0 & I & \transpose{\!S}_l & 0 \\
    0 & 0 & I & 0 \\
    0 & 0 & -\transpose{\!R}_l & I
  \end{pmatrix}
  \begin{pmatrix}
    a_lI & 0 & 0 & 0 \\
    0 & I & 0 & 0 \\
    0 & 0 & a_l^{-1}I & 0 \\
    0 & 0 & 0 & I
  \end{pmatrix}
  \begin{pmatrix}
    U_l & 0 & 0 & 0 \\
    0 & A_l & 0 & B_l \\
    0 & 0 & \transpose{U}_l^{-1} & 0 \\
    0 & C_l & 0 & D_l
  \end{pmatrix}
\end{equation}
where $R_l$ and $S_l$ are $l \times (n-l)$ matrices, $T_l$ is $l \times l$ symmetric, $a_l > 0$, $U_l \in \mathrm{GL}(l, \mathbb{R})$ with $\det U_l = \pm 1$, and $ g_l = 
\begin{pmatrix}
  A_l & B_l \\
  C_l & D_l
\end{pmatrix}
\in \mathrm{Sp}(n - l, \mathbb{R})$.
The maximal parabolic $P_n$ is the subgroup of $g \in G$ that can be written as
\begin{equation}
  \label{eq:Pndef}
    \begin{pmatrix}
    I & T_n \\
    0 & I 
  \end{pmatrix}
  \begin{pmatrix}
    a_n I & 0 \\
    0 & a_n^{-1} I 
  \end{pmatrix}
  \begin{pmatrix}
    U_n & 0 \\
    0 & \transpose{\!U_n}^{-1}
  \end{pmatrix}
\end{equation}
where $T_n$ is $n\times n$ symmetric, $a_n > 0$, and $U_n \in \mathrm{GL}(n, \mathbb{R})$ with $\det U_n = \pm 1$.
The factorizations (\ref{eq:Pldef}), (\ref{eq:Pndef}) are in fact the Langlands decompositions of $P_l$, $P_n$.
The first paper in this series \cite{MarklofWelsh2021a} contains more details on parabolic subgroups and their Langlands decompositions, and we refer the readers to \cite{Terras1988}, particularly sections 4.5.3 and 5.1, \cite{Knapp2002}, particularly section 7.7, and the authors' lecture notes \cite{MarklofWelsh2021c} for further details. 

\section{The subgroups $\Gamma$ and $\tilde{\Gamma}$}
\label{sec:Gamma}

We denote by $\Gamma$ the discrete subgroup $\Gamma = \mathrm{Sp}(n ,\mathbb{Z}) \subset G$. 
Recalling the notation of \cite{MarklofWelsh2021a}, for
\begin{equation}
  \label{eq:ABCDdef2}
  \gamma = 
  \begin{pmatrix}
    A & B \\
    C & D
  \end{pmatrix}
  \in \Gamma,
\end{equation}
we set $h_\gamma = (\bm{r}, \bm{s}, 0) \in H$ where the entries or $\bm{r}$ are $0$ or $\frac{1}{2}$ depending on whether the corresponding diagonal entry of $C\transpose{D}$ is even or odd, and the entries of $\bm{s}$ are $0$ or $\frac{1}{2}$ depending on whether the corresponding diagonal entry of $A\transpose{B}$ is even or odd.
As in \cite{MarklofWelsh2021a}, we now define the group $\tilde{\Gamma} \subset H \rtimes G$ by
\begin{equation}
  \label{eq:Gammatildedef}
  \tilde{\Gamma} = \{ ((\bm{m}, \bm{n}, t) h_\gamma ,  \gamma) \in H \rtimes G : \gamma \in \Gamma, \bm{m} \in \mathbb{Z}^n, \bm{n} \in \mathbb{Z}^n, t \in \mathbb{R}\}.
\end{equation}
The relevance of the subgroup $\tilde{\Gamma}$ is made apparent by the following theorem, see theorem 4.1 in \cite{MarklofWelsh2021a}.
\begin{theorem}
  \label{theorem:thetaautomorphy}
  For any $(uh_\gamma, \gamma) \in \tilde{\Gamma}$ and $(h,g) \in H \rtimes G$, there is a complex number $\varepsilon(\gamma)$ with $|\varepsilon(\gamma)| = 1$ such that
  \begin{equation}
    \label{eq:thetaautomorphy}
    \Theta_f( (uh_\gamma, \gamma)(h, g)) = \varepsilon(\gamma) \rho(\gamma, g) \e \left( -t + \tfrac{1}{2} \bm{m} \transpose{\bm{n}}\right) \Theta_f(h, g),
  \end{equation}
  where $u = (\bm{m}, \bm{n}, t)$.
\end{theorem}

A proof of this theorem is found in \cite{LionVergne1980} but with $\Gamma$ replaced by the finite index subgroup for which $h_\gamma = (0,0,0)$.
The automorphy under the full $\tilde{\Gamma}$ is proved in \cite{Mumford1983}, but only for the special function $f(\bm{x}) = \exp( - \pi \bm{x} \transpose{\bm{x}})$.
It is shown in \cite{LionVergne1980} that this $f$ is an eigenfunction for all the operators $R(k(Q))$, with $R$ the Segal-Shale-Weil representation and $Q \in \mathrm{U}(n)$, and it can be seen from the theory built in \cite{LionVergne1980} that the automorphy for any Schwartz function follows from that for $\exp( - \pi \bm{x} \transpose{\bm{x}})$.
A self-contained proof along the lines of \cite{LionVergne1980} is presented in the authors' lecture notes \cite{MarklofWelsh2021c}. 

\subsection{Fundamental domains}
\label{sec:fundamentaldomain}

We say that a closed set $\mathcal{D} \subset G$ is a fundamental domain for $\Gamma \backslash G$ if
\begin{itemize}
\item for all $g \in G$ there exists $\gamma \in \Gamma$ such that $\gamma g \in \mathcal{D}$ and 
\item if for $g \in \mathcal{D}$ there is a non-identity $\gamma \in \Gamma$ such that $\gamma g \in \mathcal{D}$, then $g$ is contained in the boundary of $\mathcal{D}$.
\end{itemize}
Similarly a closed set $\tilde{\mathcal{D}} \subset H \rtimes G$ is a fundamental domain for $\tilde{\Gamma} \backslash (H \rtimes G)$ if
\begin{itemize}
\item for all $(h, g) \in H \rtimes G$ there exists $\tilde{\gamma} \in \tilde{\Gamma}$ such that $\tilde{\gamma}(h, g) \in \tilde{\mathcal{D}}$ and 
\item if for $(h, g) \in \tilde{\mathcal{D}}$ there is a non-identity $\tilde{\gamma} \in \tilde{\Gamma}$ such that $\tilde{\gamma} (h,g) \in \tilde{\mathcal{D}}$, then $(h,g)$ is contained in the boundary of $\tilde{\mathcal{D}}$. 
\end{itemize}
We note that if $\mathcal{D}$ is a fundamental domain for $\Gamma \backslash G$, then
\begin{equation}
  \label{eq:tildeDdef}
  \tilde{\mathcal{D}} = \left\{ (\bm{x}, \bm{y}, 0) \in H : |x_j|, |y_j| \leq \frac{1}{2} \right\}\times \mathcal{D}
\end{equation}
is a fundamental domain for $\tilde{\Gamma} \backslash (H \rtimes G)$. 

In contrast to our previous paper \cite{MarklofWelsh2021a}, here we need to make careful use of the shape of our fundamental domain $\mathcal{D}$ in the cuspidal regions.
Drawing inspiration for the fundamental domain for $\mathrm{GL}(n, \mathbb{Z}) \backslash \mathrm{GL}(n, \mathbb{R})$ constructed in \cite{Grenier1988} as well as from the reduction theory developed in \cite{BorelJi2006} (see also \cite{Borel1969}), we construct in this section a new fundamental domain $\mathcal{D} = \mathcal{D}_n$ for $\Gamma \backslash G$.
In the following section we study the cuspidal region of $\mathcal{D}_n$.

For $n = 1$, we let $\mathcal{D}_1 \subset G$ denote the standard fundamental domain for $\Gamma \backslash G = \mathrm{SL}(2, \mathbb{Z}) \backslash \mathrm{SL}(2, \mathbb{R})$.
That is,
\begin{equation}
  \label{eq:D1def}
  \mathcal{D}_1 = \left\{
    \begin{pmatrix}
      1 & x \\
      0 & 1
    \end{pmatrix}
    \begin{pmatrix}
      y^{\frac{1}{2}} & 0 \\
      0 & y^{-\frac{1}{2}}
    \end{pmatrix}
    \begin{pmatrix}
      \cos \phi & - \sin \phi \\
      \sin \phi & \cos \phi
    \end{pmatrix}
    : |x| \leq \frac{1}{2}, x^2 + y^2 \geq 1, 0\leq \phi < 2\pi\right\}.
\end{equation}
We now define fundamental domains $\mathcal{D}_n$ inductively using the decomposition (\ref{eq:Pjdecomp}) for $l =1$. 
Writing $g \in G$ as
\begin{equation}
  \label{eq:gXUVkQ}
  g =
  \begin{pmatrix}
    1 & \bm{r}_1 & t_1 - \bm{s}_1 \transpose{\bm{r}}_1 & \bm{s}_1 \\
    0 & I & \transpose{\bm{s}}_1 & 0 \\
    0 & 0 & 1 & 0 \\
    0 & 0 & -\transpose{\bm{r}}_1 & I    
  \end{pmatrix}
  \begin{pmatrix}
    1 & 0 & 0 & 0 \\
    0 & I & 0 & X_1 \\
    0 & 0 & 1 & 0 \\
    0 & 0 & 0 & I
  \end{pmatrix}
  \begin{pmatrix}
    v_1^{\frac{1}{2}} & 0 & 0 & 0 \\
    0 & Y_1^{\frac{1}{2}} & 0 & 0 \\
    0 & 0 & v_1^{-\frac{1}{2}} & 0 \\
    0 & 0 & 0 & \transpose{Y_1}^{-\frac{1}{2}}
  \end{pmatrix}
  k(Q),
\end{equation}
where $\bm{r} = \bm{r}(g) \in \mathbb{R}^{n-1}$, $\bm{s} = \bm{s}(g) \in \mathbb{R}^{n-1}$, $t_1 = t_1(g) \in \mathbb{R}$, $X_1 = X_1(g)$ is symmetric, $v_1 = v_1(g) > 0$, $Y_1 = Y_1(g)$ is positive definite symmetric, and $Q \in \mathrm{U}(n)$,
we define $\mathcal{D}_n$ as the set of all $g \in G$ satisfying
\begin{itemize}
\item $v_1(g) \geq v_1(\gamma g)$ for all $\gamma \in \Gamma$, 
\item $g_1(g) \in \mathcal{D}_{n-1}$, see (\ref{eq:gldef}), and 
\item the entries of $\bm{r}_1(g)$, $\bm{s}_1(g)$, and $t_1(g)$ are all less than or equal to $\frac{1}{2}$ in absolute value with the first entry of $\bm{r}_1$ greater than or equal to $0$.
\end{itemize}

\begin{proposition}
  \label{proposition:fundamentaldomain}
  $\mathcal{D}_n$ is a fundamental domain for $\Gamma \backslash G$.
\end{proposition}

\begin{proof}
  We begin by showing that for $g \in G$, $\sup_{\gamma \in \Gamma} v_1(\gamma g)$ is indeed obtained by some $\gamma \in \Gamma$.
  From (\ref{eq:XYQ}), we have for
  \begin{equation}
    \label{eq:gammaform}
    \gamma =
    \begin{pmatrix}
      A & B \\
      C & D
    \end{pmatrix}
    \in \Gamma
  \end{equation}
  that
  \begin{equation}
    \label{eq:V1gammag}
    v_1(\gamma g)^{-1} =  \bm{c} Y \transpose{\bm{c}} +  ( \bm{c} X + \bm{d}) Y^{-1} \transpose{(\bm{c}X + \bm{d})}
  \end{equation}
  where
  \begin{equation}
    \label{eq:gXY}
    g =
    \begin{pmatrix}
      I & X \\
      0 & I 
    \end{pmatrix}
    \begin{pmatrix}
      Y^{\frac{1}{2}} & 0 \\
      0 & \transpose{Y}^{-\frac{1}{2}}
    \end{pmatrix}
    k(Q)
  \end{equation}
  and $\bm{c}$, $\bm{d}$ are the first rows of $C$, $D$.
  Since $Y$ is positive definite, there are only finitely many $\bm{c}$ such that $\bm{c} Y \transpose{\bm{c}}$, and hence $v_1(\gamma g)^{-1}$, is below a given bound.
  Similarly, for a fixed $\bm{c}$, the positive definiteness of $Y^{-1}$ implies that there are only finitely many $\bm{d}$ such that $v_1(\gamma g)^{-1}$ is below a given bound.
  It follows that there are only finitely many $\gamma \in \Gamma_1 \backslash \Gamma$ such that $v_1(\gamma g)$ is larger than a given bound, where $\Gamma_1 = \Gamma \cap P_1$ and we recall $P_1$ is given by (\ref{eq:Pldef}). 
  As $v_1(\gamma g) = v_1(g)$ for $\gamma \in \Gamma_1$ it follows that $v_1(\gamma g)$ is maximized for some $\gamma \in \Gamma$.

  Let $\gamma_0$ be so that $v_1(\gamma_0 g)$ is maximal.
  We now decompose an arbitrary $\gamma \in \Gamma_1$ as in (\ref{eq:Pldef}),
  \begin{equation}
    \label{eq:GammaPdecomp}
    \gamma =
    \begin{pmatrix}
      1 & \bm{r}_1 & t_1 - \bm{s}_1 \transpose{\bm{r}_1} & \bm{s}_1 \\
      0 & I & \transpose{\bm{s}_1} & 0 \\
      0 & 0 & 1 & 0 \\
      0 & 0 & -\transpose{\bm{r}_1} & I
    \end{pmatrix}
    \begin{pmatrix}
      \pm 1 & 0 & 0 & 0 \\
      0 & A_1 & 0 & B_1 \\
      0 & 0 & \pm 1 & 0 \\
      0 & C_1 & 0 & D_1
    \end{pmatrix}    
  \end{equation}
  with
  \begin{equation}
    \label{eq:gamma1def}
    \gamma_1 =
    \begin{pmatrix}
      A_1 & B_1 \\
      C_1 & D_1 
    \end{pmatrix}
    \in \mathrm{Sp}(n-1, \mathbb{Z}). 
  \end{equation}
  Proceeding inductively, there exists $\gamma_1$ such that $\gamma_1 g_1(\gamma_0g) = g_1(\gamma \gamma_0 g) \in \mathcal{D}_{n-1}$.
  Now, we can change $\bm{r}_1(\gamma)$, $\bm{s}_1(\gamma)$, $t_1(\gamma)$, and the $\pm$, noting that this does not change $g_1(\gamma \gamma_0 g)$, so that the entries of $\bm{r}_1(\gamma \gamma_0 g)$, $\bm{s}_1(\gamma \gamma_0 g)$ and $t_1(\gamma \gamma_0 g)$ are all $\leq \frac{1}{2}$ in absolute value and the first entry of $\bm{r}_1(\gamma \gamma_0 g)$ is nonnegative.
  Therefore $\gamma \gamma_1 g \in \mathcal{D}_n$ as required.

  We now suppose that $g \in \mathcal{D}_n$ and there is a non-identity $\gamma \in \Gamma$ such that $\gamma g \in \mathcal{D}_n$.
  We set
  \begin{equation}
    \label{eq:gammag}
    \gamma =
    \begin{pmatrix}
      A & B \\
      C & D
    \end{pmatrix}
    ,\quad g =
    \begin{pmatrix}
      I & X \\
      0 & I
    \end{pmatrix}
    \begin{pmatrix}
      Y^{\frac{1}{2}} & 0 \\
      0 & \transpose{Y}^{-\frac{1}{2}} 
    \end{pmatrix}
    k(Q).
  \end{equation}
  By the maximality, we have $v_1(g) = v_1(\gamma g)$ and therefore
  \begin{equation}
    \label{eq:V1equality}
    v_1^{-1} = \bm{c} Y \transpose{\bm{c}} + (\bm{c}X + \bm{d}) Y^{-1} \transpose{(\bm{c}X + \bm{d})}
  \end{equation}
  where $\bm{c}$ and $\bm{d}$ are the first rows of $C$ and $D$.
  Let us first consider the case when $\bm{c} \neq 0$.
  To show that $g$ is on the boundary of $\mathcal{D}_n$ in this case, we consider
  \begin{equation}
    \label{eq:gepsilon}
    g_\epsilon =
    \begin{pmatrix}
      I & X \\
      0 & I
    \end{pmatrix}
    \begin{pmatrix}
      (1 - \epsilon)^{\frac{1}{2}} Y^{\frac{1}{2}} & 0 \\
      0 & (1 - \epsilon)^{-\frac{1}{2}} \transpose{Y}^{-\frac{1}{2}}
    \end{pmatrix}
    k(Q)
  \end{equation}
  for $0 < \epsilon < 1$.
  We have $v_1(g_\epsilon) = (1 - \epsilon) v_1(g)$ and
  \begin{multline}
    \label{eq:V1gammagepsilon}
    v_1(\gamma g_\epsilon)^{-1} = (1 - \epsilon) \bm{c}Y \transpose{\bm{c}} + (1 - \epsilon)^{-1} (\bm{c} X + \bm{d}) Y^{-1} \transpose{(\bm{c}X + \bm{d})} \\
    = \left( (1 - \epsilon) - (1 - \epsilon)^{-1} \right) \bm{c}Y \transpose{\bm{c}} + v_1(g_\epsilon)^{-1}
  \end{multline}
  by (\ref{eq:V1equality}). 
  Since $v_1(\gamma g_\epsilon) > v_1(g_\epsilon)$, we have that $g_\epsilon \not\in \mathcal{D}_n$.
  As $g_\epsilon$ can be made arbitrarily close to $g$, we conclude that $g$ is on the boundary of $\mathcal{D}_n$.

  If $\bm{c} = 0$, then from (\ref{eq:V1equality}) we have
  \begin{equation}
    \label{eq:V1equality1}
    v_1(g)^{-1} = (d^{(1)} - \bm{d}^{(2)} \transpose{\bm{r}}_1)^2 v_1(g)^{-1} + \bm{d}^{(2)} Y_1^{-1} \transpose{\bm{d}}^{(2)}
  \end{equation}
  where $\bm{d} = \begin{pmatrix} d^{(1)} & \bm{d}^{(2)} \end{pmatrix}$ are as above
  \begin{equation}
    \label{eq:V1Y11}
    Y =
    \begin{pmatrix}
      1 & \bm{r}_1 \\
      0 & I
    \end{pmatrix}
    \begin{pmatrix}
      v_1 & 0 \\
      0 & Y_1
    \end{pmatrix}
    \begin{pmatrix}
      1 & 0 \\
      -\transpose{\bm{r}_1} & I
    \end{pmatrix}
    .
  \end{equation}
  This time we consider
  \begin{equation}
    \label{eq:gepsilon1}
    g_\epsilon =
    \begin{pmatrix}
      I & X \\
      0 & I
    \end{pmatrix}
    \begin{pmatrix}
      Y_\epsilon^{\frac{1}{2}} & 0 \\
      0 & \transpose{Y_\epsilon}^{-\frac{1}{2}}
    \end{pmatrix}
    k(Q)
  \end{equation}
  with
  \begin{equation}
    \label{eq:Yepsilon}
    Y_\epsilon =
    \begin{pmatrix}
      1 & \bm{r}_1 \\
      0 & I
    \end{pmatrix}
    \begin{pmatrix}
      (1 - \epsilon) v_1 & 0 \\
      0 & Y_1
    \end{pmatrix}
    \begin{pmatrix}
      1 & 0 \\
      -\transpose{\bm{r}_1} & I
    \end{pmatrix}
    .
  \end{equation}
  We have $v_1(g_\epsilon) = (1 - \epsilon)v_1(g)$ and
  \begin{multline}
    \label{eq:V1gammagepsilon1}
    v_1(\gamma g_\epsilon)^{-1} = (1 - \epsilon)^{-1} (d^{(1)} - \bm{d}^{(2)} \transpose{\bm{r}}_1)^2 v_1(g)^{-1} + \bm{d}^{(2)} Y_1^{-1} \transpose{\bm{d}}^{(2)} \\
    = v_1(g_\epsilon)^{-1} + \left( 1 - ( 1 - \epsilon)^{-1} \right) \bm{d}^{(2)} Y_1 \transpose{\bm{d}^{(2)}}
  \end{multline}
  from (\ref{eq:V1equality1}).
  If $\bm{d}^{(2)} \neq 0$, then $v_1(\gamma g_\epsilon) > v_1(g_\epsilon)$ and we conclude that $g$ is on the boundary of $\mathcal{D}_n$ as before.

  When $\bm{c} = 0$ and $\bm{d}^{(2)} = 0$ we have $d^{(1)} = \pm 1$, and so $\gamma \in \Gamma_1$.
  We decompose $\gamma$ as in (\ref{eq:GammaPdecomp}) and define $\gamma_1$ as in (\ref{eq:gamma1def}).
  By the construction of $\mathcal{D}_n$, we have $g_1(g) \in \mathcal{D}_{n-1}$ and $ g_1(\gamma g) = \gamma_1 g_1(g) \in \mathcal{D}_{n-1}$.
  By induction, we have that either $\gamma_1$ is the identity or $g_1(g)$ is on the boundary of $\mathcal{D}_{n-1}$.
  In the latter case we have that $g$ is on the boundary of $\mathcal{D}_n$, and so it remains to consider
  \begin{equation}
    \label{eq:lastgammaform}
    \gamma =
    \begin{pmatrix}
      \pm 1 & \bm{r}_1 & \pm t_1 \mp \bm{r}_1\transpose{\bm{s}_1} & \bm{s}_1 \\
      0 & I & \pm \transpose{\bm{s}_1} & 0 \\
      0 & 0 & \pm 1 & 0 \\
      0 & 0 & \mp \bm{r}_1 & I
    \end{pmatrix}
    .
  \end{equation}

  If any of the entries of $\bm{r}_1(\gamma)$ or $\bm{s}_1(\gamma)$ is not zero, then the corresponding entry of $\bm{r}_1(g)$ or $\bm{s}_1(g)$ is $\pm \frac{1}{2}$ and so $g$ is on the boundary of $\mathcal{D}_n$.
  Similarly if $t_1(\gamma) \neq 0$, we have $t_1(g) = \pm \frac{1}{2}$ and again $g$ is on the boundary of $\mathcal{D}_n$.
  If all of $\bm{r}_1, \bm{s}_1, t_1$ are $0$, the sign must be $-$ as $\gamma$ is not the identity, and it follows that the first entry of $\bm{r}_1(g)$ is $0$ and $g$ is again on the boundary of $\mathcal{D}_n$. 
\end{proof}

The following proposition records some useful properties of $\mathcal{D}_n$.
It and its proof are very similar to the analogous statement for the different fundamental domain used in \cite{MarklofWelsh2021a}, see  proposition 3.1 there.  
\begin{proposition}
  \label{proposition:Dnproperties}
  Let $g \in \mathcal{D}_n$ and write
  \begin{equation}
    \label{eq:gVrewrite}
    g =
    \begin{pmatrix}
      I & X \\
      0 & I
    \end{pmatrix}
    \begin{pmatrix}
      Y^{\frac{1}{2}} & 0 \\
      0 & Y^{-\frac{1}{2}}
    \end{pmatrix}
    k(Q), \quad Y = U V \transpose{U}
    ,
  \end{equation}
  where $X$ is symmetric, $Y$ is positive definite symmetric, $U$ upper triangular unipotent, $V$ positive diagonal, and $Q \in \mathrm{U}(n)$, and
  \begin{equation}
    \label{eq:Yrewrite}
    V =
    \begin{pmatrix}
      v_1 & \cdots & 0 \\
      \vdots & \ddots & \vdots \\
      0 & \cdots & v_n
    \end{pmatrix}
    ,\quad Y =
    \begin{pmatrix}
      1 & \bm{r}_1 \\
      0 & I
    \end{pmatrix}
    \begin{pmatrix}
      v_1 & 0 \\
      0 & Y_1
    \end{pmatrix}
    \begin{pmatrix}
      1 & 0 \\
      \transpose{\bm{r}_1} & I
    \end{pmatrix}
    .
  \end{equation}
  Then we have
  \begin{enumerate}
  \item $v_n \geq \frac{\sqrt{3}}{2}$ and $v_j \geq \frac{3}{4} v_{j+1}$ for $1\leq j \leq n-1$, 
  \item for all $\bm{x} =
    \begin{pmatrix}
      x^{(1)} & \bm{x}^{(2)}
    \end{pmatrix}
    \in \mathbb{R}^n$
    \begin{equation}
      \label{eq:approxortho}
      \bm{x} Y \transpose{\bm{x}} \asymp_n v_1 (x^{(1)})^2 + \bm{x}^{(2)} Y_1 \transpose{\bm{x}}^{(2)}.
    \end{equation}
  \end{enumerate}
\end{proposition}

\begin{proof}
  For the first, we observe that by the inductive construction of $\mathcal{D}_n$, we have that
  \begin{equation}
    \label{eq:D1element}
    g_{n-1}(g) = 
    \begin{pmatrix}
      1 & x_{n-1}(g) \\
      0 & 1
    \end{pmatrix}
    \begin{pmatrix}
      v_n^{\frac{1}{2}} & 0 \\
      0 & v_n^{-\frac{1}{2}}
    \end{pmatrix}
    \in \mathcal{D}_1.
  \end{equation}
  As $\mathcal{D}_1$ is the standard fundamental domain for $\mathrm{SL}(2, \mathbb{Z}) \backslash \mathrm{SL}(2, \mathbb{R})$, we conclude that $v_n \geq \frac{\sqrt{3}}{2}$.

  To demonstrate that $v_j \geq \frac{3}{4} v_{j+1}$, we note that by the construction of $\mathcal{D}_n$, it suffices to consider only $j = 1$.
  We start with
  \begin{equation}
    \label{eq:V1maximality}
    v_1^{-1} \leq \bm{c} Y \transpose{\bm{c}} + (\bm{c}X + \bm{d}) Y^{-1} \transpose{( \bm{c}X + \bm{d})}
  \end{equation}
  for any $\begin{pmatrix} \bm{c} & \bm{d} \end{pmatrix} \in \mathbb{Z}^{2n}$ nonzero and primitive.
  Choosing $\bm{c} = 0$ and $\bm{d} = \begin{pmatrix} 0 & 1 & 0 \cdots & 0 \end{pmatrix}$, we have
  \begin{equation}
    \label{eq:v1bound}
    v_1^{-1} \leq v_1^{-1} ( r_1^{(1)})^2 + v_2^{-1},
  \end{equation}
  where $r_1^{(1)}$ is the first entry of $\bm{r}_1$.
  Since $0\leq  r_1^{(1)} \leq \frac{1}{2}$, we conclude that $v_1 \geq \frac{3}{4} v_2$.

  To demonstrate the second part of the proposition, we let $\bm{y}_1, \dots, \bm{y}_n$ denote the rows of
  \begin{equation}
    \label{eq:Ysquareroot}
    Y^{\frac{1}{2}} =
    \begin{pmatrix}
      1 & \bm{r}_1 \\
      0 & I 
    \end{pmatrix}
    \begin{pmatrix}
      v_1^{\frac{1}{2}} & 0 \\
      0 & Y_1^{\frac{1}{2}}
    \end{pmatrix}
    .
  \end{equation}
  Setting $\bm{y} = x_2 \bm{y}_2 + \cdots + x_n \bm{y}_n$, where the $x_j$ are the entries of $\bm{x}$, our aim is to prove that for some constants $0 < c_1 < 1 < c_2$ depending only on $n$,
  \begin{equation}
    \label{eq:rbound}
    c_1 \left( ||\bm{y}_1||^2 x_1^2 + || \bm{y} ||^2 \right) \leq || x_1 \bm{y}_1 + \bm{y} ||^2 \leq c_2 \left( ||\bm{y}_1||^2 x_1^2 + || \bm{y} ||^2 \right) ,
  \end{equation}
  from which the lower bound in (\ref{eq:approxortho}) follows as $|| \bm{y}_1 ||^2 \geq v_1$.
  The upper bound in (\ref{eq:approxortho}) follows from (\ref{eq:rbound}) and $v_1 \gg || \bm{y}_1 ||^2$, which is verified below, see (\ref{eq:r1normbound}). 
  Expanding the expression in the middle of (\ref{eq:rbound}), we find that it is enough to show that
  \begin{equation}
    \label{eq:leftsideexpand}
    2 | x_1 \bm{y}_1 \transpose{\!\bm{y}} | \leq (1 - c_1) \left( || \bm{y}_1||^2  x_1^2 + || \bm{y} ||^2 \right),
  \end{equation}
  and
  \begin{equation}
    \label{eq:middleexpand2}
    2 | x_1 \bm{y}_1 \transpose{\!\bm{y}} | \leq (c_2 -1) \left( || \bm{y}_1||^2  x_1^2 + || \bm{y} ||^2 \right).  
  \end{equation}
  The upper bound (\ref{eq:middleexpand2}) is trivial if $c_2 = 2$, and the upper bound (\ref{eq:leftsideexpand}) would follow from
  \begin{equation}
    \label{eq:weakerinequality}
    | \bm{y}_1 \transpose{\!\bm{y}} | \leq (1 - c_1) || \bm{y}_1||\; || \bm{y} ||.
  \end{equation}

  We let $0 < \phi_1 < \pi$ denote the angle between $\bm{y}_1$ and $\bm{y}$ and $ 0 < \phi_2 < \frac{\pi}{2}$ denote the angle between $\bm{y}_1$ and the hyperplane $\mathrm{span} (\bm{y}_2, \dots, \bm{y}_n)$.
  We have $\phi_2 \leq \mathrm{min}(\phi_1, \pi - \phi_1)$, and so $| \cos \phi_1 | \leq | \cos \phi_2 |$.
  We bound $ \cos \phi_2$ away from $1$ by bounding $\sin \phi_2$ away from $0$.

  We have
  \begin{equation}
    \label{eq:sinphi2bound}
    |\sin \phi_2 | = \frac{|| \bm{y}_1 \wedge \cdots \wedge \bm{y}_n || }{|| \bm{y}_1 ||\; || \bm{y}_2 \wedge \cdots \wedge \bm{y}_n ||} = \frac{v_1^{\frac{1}{2}}}{|| \bm{y}_1 ||},
  \end{equation}
  so it suffices to show that $v_1^{\frac{1}{2}} \gg || \bm{y}_1 ||$.
  Here $\wedge$ denotes the usual wedge product on $\mathbb{R}^n$ and the norm on $\bigwedge^k \mathbb{R}^n$ is given by
  \begin{equation}
    \label{eq:wedgenorm}
    || \bm{a}_1 \wedge \cdots \wedge \bm{a}_k ||^2 = \det
    \begin{pmatrix}
      \bm{a}_1 \\
      \vdots \\
      \bm{a}_k
    \end{pmatrix}
    \begin{pmatrix}
      \transpose{\bm{a}}_1 & \cdots & \transpose{\bm{a}}_k
    \end{pmatrix}
    .
  \end{equation}
  Using the inductive construction of $\mathcal{D}_n$ and the fact that the entries of $\bm{r}_1(Y), \bm{r}_1(Y_1), \dots$ are at most $\frac{1}{2}$ in absolute value, we observe that $U$ has entries bounded by a constant depending only on $n$.
  We find that
  \begin{equation}
    \label{eq:r1normbound}
    || \bm{y}_1 ||^2 \ll v_1 + \cdots + v_n \ll v_1
  \end{equation}
  with the implied constant depending on $n$. 
\end{proof}

\subsection{Shape of the cusp}
\label{sec:cuspshape}

As explicated in \cite{Borel1969} and \cite{BorelJi2006}, the cusp of $\Gamma \backslash G$ can be partitioned into $2^n -1$ box-shaped regions.
These regions are in correspondence with the conjugacy classes of proper parabolic subgroups of $G$ and are formed as $K$ times the product of three subsets, one for each of the components -- nilpotent, diagonal, and semisimple -- of the Langlands decomposition of $P$.

In what follows we use the fundamental domain $\mathcal{D}_n$ constructed in section \ref{sec:fundamentaldomain} to prove a variation of this fact, although only for the maximal parabolic subgroups (\ref{eq:Pldef}), (\ref{eq:Pndef}).
Our main result for this section is proposition \ref{proposition:vllarge}, which roughly states that if $g \in G$ is close enough the boundary in a precise sense, then $g$ can be brought into $\mathcal{D}_n$ by an element $\gamma$ in some maximal parabolic subgroup which depends on the way $g$ approaches the boundary.

For $1\leq l < n$ we denote by $\Gamma_{l,1}$ and $\Gamma_{l,2}$ the subgroups of $\Gamma_l = \Gamma \cap P_l$ given by
\begin{equation}
  \label{eq:Gammal1def}
  \Gamma_{l,1} = \left\{
    \begin{pmatrix}
      A & 0 & 0 & 0 \\
      0 & I & 0 & 0 \\
      0 & 0 & \transpose{A}^{-1} & 0 \\
      0 & 0 & 0 & I
    \end{pmatrix}
    : A \in \mathrm{GL}(l, \mathbb{Z}) \right\}
\end{equation}
and
\begin{equation}
  \label{eq:Gammal2def}
  \Gamma_{l,2} = \left\{
    \begin{pmatrix}
      I & 0 & 0 & 0 \\
      0 & A & 0 & B \\
      0 & 0 & I & 0 \\
      0 & C & 0 & D
    \end{pmatrix}
    :
    \begin{pmatrix}
      A & B \\
      C & D
    \end{pmatrix}
    \in \mathrm{Sp}(n - l, \mathbb{Z}) \right\}. 
\end{equation}
For $l =n$, we set
\begin{equation}
  \label{eq:Gamman1def}
  \Gamma_{n,1} = \left\{
    \begin{pmatrix}
      A & 0 \\
      0 & \transpose{A}^{-1}
    \end{pmatrix}
    : A \in \mathrm{GL}(n, \mathbb{Z}) \right\},
\end{equation}
and we let $\Gamma_{n,2}$ be trivial.
We now define for $g \in G$ and $1\leq l \leq n$,
\begin{equation}
  \label{eq:vlGammaldef}
  v_l( \Gamma_l g) = \min_{\gamma \in \Gamma_l} v_l(\gamma g) = \min_{\gamma \in \Gamma_{l,1}} v_l(\gamma g)
\end{equation}
and, for $1\leq l < n$,
\begin{equation}
  \label{eq:vl1Gammaldef}
  v_{l+1}( \Gamma_l g) = \max_{\gamma \in \Gamma_l} v_{l+1}(\gamma g) = \max_{\gamma \in \Gamma_{l,2}} v_{l+1}( \gamma g).
\end{equation}
We note that in the proof of proposition \ref{proposition:fundamentaldomain}, we saw that the maximum in (\ref{eq:vl1Gammaldef}) does exist.
As for the minimum in (\ref{eq:vlGammaldef}), we simply note that
\begin{equation}
  \label{eq:vlAYA}
  v_l(A U_l V_l \transpose{U}_l\transpose{A}) = \bm{a} U_l V_l \transpose{U_l} \transpose{\bm{a}}
\end{equation}
where $\bm{a}$ is the last row of $A \in \mathrm{GL}(l, \mathbb{Z})$, so the positive definiteness of $U_l V_l \transpose{U}_l$ implies that there are only finitely many values of $v_l(A U_l V_l \transpose{U}_l \transpose{A})$ below a given bound. 

We now define a fundamental domain $\mathcal{D}_l'$ for the action of $\mathrm{GL}(l, \mathbb{Z})$ on $l\times l$ positive definite symmetric matrices.
We set $\mathcal{D}_1' = \{ y > 0 \}$ and
\begin{equation}
  \label{eq:D2primedef}
  \mathcal{D}_2' = \left\{
    \begin{pmatrix}
      1 & r \\
      0 & 1 
    \end{pmatrix}
    \begin{pmatrix}
      v_1 & 0 \\
      0 & v_2 
    \end{pmatrix}
    \begin{pmatrix}
      1 & 0 \\
      r & 1
    \end{pmatrix}
    : 0 \leq r \leq \frac{1}{2},\ r^2 + \frac{v_1}{v_2} \geq 1 \right\},
\end{equation}
the standard fundamental domain for $\mathrm{GL}(2, \mathbb{Z})$ acting on $2\times 2$ positive definite symmetric matrices.
The domain $\mathcal{D}_l'$ for $l > 2$ is then defined inductively as the set of all
\begin{equation}
  \label{eq:UVUdecomp}
  Y = 
  \begin{pmatrix}
    1 & \bm{r} \\
    0 & I
  \end{pmatrix}
  \begin{pmatrix}
    v_1 & 0 \\
    0 & Y_1
  \end{pmatrix}
  \begin{pmatrix}
    1 & 0 \\
    \bm{r} & 1
  \end{pmatrix}
\end{equation}
such that
\begin{enumerate}
\item $v_1(Y) \geq v_1( A Y \transpose{A})$ for all $A \in \mathrm{GL}(l, \mathbb{Z})$,
\item $Y_1 \in \mathcal{D}_{l-1}'$, and 
\item $|r_j | \leq \frac{1}{2}$ and $0 \leq r_1 \leq \frac{1}{2}$ where $r_j$ are the entries of $\bm{r}$. 
\end{enumerate}
This is in fact the set of $Y$ such that $Y^{-1}$ is in Grenier's fundamental domain, see \cite{Grenier1988} and \cite{Terras1988}, so we do not prove that $\mathcal{D}_l'$ is a fundamental domain here.
We do however record the following properties of $\mathcal{D}_{l}'$.

\begin{lemma}
  \label{lemma:Grenier}
  Let $U V \transpose{U} \in \mathcal{D}_l'$ with
  \begin{equation}
    \label{eq:Vvj}
    V =
    \begin{pmatrix}
      v_1 & \cdots & 0 \\
      \vdots & \ddots & \vdots \\
      0 & \cdots & v_l
    \end{pmatrix}
  \end{equation}
  positive diagonal and $U$ upper triangular unipotent.
  Then we have
  \begin{enumerate}
  \item $v_j \geq \frac{3}{4} v_{j+1}$ for $1\leq j < l$,
  \item for any $\bm{x} \in \mathbb{R}^l$,
    \begin{equation}
      \label{eq:DprimeYbound}
      \bm{x} U V \transpose{U} \transpose{\bm{x}} \asymp \bm{x} V \transpose{\bm{x}}
    \end{equation}
    with implied constant depending only on $l$, and
  \item
    \begin{equation}
      \label{eq:minvl}
      \min_{A \in \mathrm{GL}(l, \mathbb{Z})} v_l ( A U V \transpose{U} \transpose{A}) \asymp v_l(UV\transpose{U})
    \end{equation}
    with implied constant depending only on $l$. 
  \end{enumerate}
\end{lemma}

\begin{proof}
  The first and second parts are proved in proposition 3.1 of \cite{MarklofWelsh2021a}.
  To prove the third part, we note that with $\bm{a}$ the last row of $A$,
  \begin{equation}
    \label{eq:vlbound}
    v_l( A U V \transpose{U} \transpose{A}) = \bm{a} U V \transpose{U} \transpose{\bm{a}} \gg \bm{a} V \transpose{\bm{a}},
  \end{equation}
  by the second part of the lemma.
  Applying the first part of the lemma we have $\bm{a} V \transpose{\bm{a}} \gg v_l || \bm{a} ||^2 \geq v_l$, and (\ref{eq:minvl}) follows. 
\end{proof}

As the proof is almost identical to the proof of the third part of lemma \ref{lemma:Grenier}, we record the following lemma for later use.

\begin{lemma}
  \label{lemma:Dminvl}
  If $g \in \mathcal{D}_n$, then for all $1\leq l < n$,
  \begin{equation}
    \label{eq:Dminvl}
    v_l ( \Gamma_l g) \asymp v_l(g)
  \end{equation}
  with the implied constant depending only on $n$. 
\end{lemma}

\begin{proof}
  We recall from the second part of proposition \ref{proposition:Dnproperties} that for $\bm{x} \in \mathbb{R}^l$,
  \begin{equation}
    \label{eq:xUlVl}
    \bm{x} U_l V_l \transpose{U}_l \transpose{\bm{x}} \gg \bm{x} V_l \transpose{\bm{x}}. 
  \end{equation}
  We have
  \begin{equation}
    \label{eq:vlGammalbound}
    v_l(\Gamma_l g) = \min_{\substack{\bm{c} \in \mathbb{Z}^l \\ \bm{c} \neq 0}}  \bm{c} U_l V_l \transpose{U}_l \transpose{\bm{c}} \gg \min_{\substack{\bm{c} \in \mathbb{Z}^l \\ \bm{c} \neq 0}} \bm{c} V_l \transpose{\bm{c}}. 
  \end{equation}
  Now as $\bm{c} \neq 0$, we have $c_j^2 \geq 1$ for some $1\leq j\leq l$, and so
  \begin{equation}
    \label{eq:vlGammalbound1}
    v_l(\Gamma_l g) \gg v_j(g) \gg v_l(g)
  \end{equation}
  by the first part of proposition \ref{proposition:Dnproperties}.
\end{proof}

We are now ready to prove the main result for this section.

\begin{proposition}
  \label{proposition:vllarge}
  For $1\leq l \leq n$, there are constants $a_l > 0 $ such that for $l < n$, if $g \in G$ satisfies $v_l(\Gamma_l g) \geq a_l v_{l+1}(\Gamma_l g)$, and for $l=n$ if $g \in G$ satisfies $v_n(\Gamma_n g) \geq a_n$, then there exists $\gamma \in \Gamma_l$ so that $\gamma g \in \mathcal{D}_n$.
  Moreover, for this $\gamma$ we have $v_l(\Gamma_l g) \asymp v_l(\gamma g)$ and, for $l < n$, $v_{l+1}( \Gamma_{l} g) = v_{l+1}(\gamma g)$. 
\end{proposition}

We remark that this proposition can be  extended to any of the parabolic subgroups $P_L$ of $G$ by taking intersections of the maximal parabolics.
However some care needs to be taken regarding the possible non-uniqueness of the $\gamma$ bringing $g$ into $\mathcal{D}_n$.
Since it is unnecessary for our goals, we do not discuss this here.

\begin{proof}
  By multiplying $g$ by
  \begin{equation}
    \label{eq:gamma1form}
    \gamma_1 =
    \begin{pmatrix}
      A' & 0 & 0 & 0 \\
      0 & A & 0 & B \\
      0 & 0 & \transpose{(A')}^{-1} & 0 \\
      0 & C & 0 & D
    \end{pmatrix}
    \in \Gamma_l,
  \end{equation}
  we may assume that $U_l V_l \transpose{U}_l \in \mathcal{D}_l'$ and
  \begin{equation}
    \label{eq:XlYlDnl}
    \begin{pmatrix}
      I & X_l \\
      0 & I 
    \end{pmatrix}
    \begin{pmatrix}
      Y_l^{\frac{1}{2}} & 0 \\
      0 & \transpose{Y_l}^{-\frac{1}{2}}
    \end{pmatrix}
    \in \mathcal{D}_{n-l}.
  \end{equation}
  We recall that for $\gamma =
  \begin{pmatrix}
    A & B \\
    C & D
  \end{pmatrix}
  $,
  \begin{equation}
    \label{eq:v1gammag}
    v_1(\gamma g)^{-1} = \bm{c} Y \transpose{\bm{c}} + ( \bm{c}X + \bm{d}) Y^{-1} \transpose{ ( \bm{c} X + \bm{d})}
  \end{equation}
  where $\bm{c}$, $\bm{d}$ are the first rows of $C$, $D$.
  Now, writing $\bm{c} =
  \begin{pmatrix}
    \bm{c}^{(1)} & \bm{c}^{(2)}
  \end{pmatrix}
  $, $\bm{d} =
  \begin{pmatrix}
    \bm{d}^{(1)} & \bm{d}^{(2)}
  \end{pmatrix}
  $ and
  \begin{align}
    \label{eq:XYlcoordinates}
    & X =
    \begin{pmatrix}
      T_l + R_l X_l \transpose{R}_l & S_l + R_l X_l \\
      \transpose{S}_l + X_l \transpose{R}_l & X_l
    \end{pmatrix}
    , \\
    & Y =
    \begin{pmatrix}
      U_l & R_l \\
      0 & I
    \end{pmatrix}
    \begin{pmatrix}
      V_l & 0 \\
      0 & Y_l
    \end{pmatrix}
    \begin{pmatrix}
      \transpose{U}_l & 0 \\
      \transpose{R}_l & I
    \end{pmatrix}
    ,
  \end{align}
  see (\ref{eq:Pjdecomp}), we obtain
  \begin{align}
    \label{eq:v1gammaglcoordinates}
    v_1(\gamma g)^{-1} = & \bm{c}^{(1)} U_l V_l \transpose{U}_l \transpose{\bm{c}}^{(1)} + ( \bm{c}^{(1)} R_l + \bm{c}^{(2)}) Y_l \transpose{( \bm{c}^{(1)} R_l + \bm{c}^{(2)})} \nonumber \\
                         & + \left( \bm{c}^{(1)} ( T_l - S_l \transpose{R}_l) + \bm{c}^{(2)} \transpose{S}_l + \bm{d}^{(1)} - \bm{d}^{(2)} \transpose{R}_l\right) \transpose{U}_l^{-1} V_l^{-1} U_l^{-1} \nonumber \\
                         & \qquad \transpose{\left( \bm{c}^{(1)} ( T_l - S_l \transpose{R}_l) + \bm{c}^{(2)} \transpose{S}_l + \bm{d}^{(1)} - \bm{d}^{(2)} \transpose{R}_l\right)} \nonumber \\
                         & + \left( \bm{c}^{(1)} ( S_l + R_l X_l) + \bm{c}^{(2)} X_l + \bm{d}^{(2)} \right) Y_l^{-1} \nonumber \\
                         & \qquad \transpose{\left( \bm{c}^{(1)} ( S_l + R_l X_l) + \bm{c}^{(2)} X_l + \bm{d}^{(2)} \right)}.
  \end{align}

  If $\bm{c}^{(1)} \neq 0$, then, since $U_l V_l \transpose{U}_l \in \mathcal{D}_l'$, we have
  \begin{equation}
    \label{eq:vlmin}
    v_1(\gamma g)^{-1} \geq \bm{c}^{(1)} U_l V_l \transpose{U}_l \transpose{\bm{c}}^{(1)} \gg \bm{c}^{(1)} V_l \transpose{\bm{c}}^{(1)} \gg  v_l
  \end{equation}
  by the second part of lemma \ref{lemma:Grenier}.    
  Since, for $l < n$,
  \begin{equation}
    \label{eq:XlYlDnl1}
    \begin{pmatrix}
      I & X_l \\
      0 & I 
    \end{pmatrix}
    \begin{pmatrix}
      Y_l^{\frac{1}{2}} & 0 \\
      0 & \transpose{Y_l}^{-\frac{1}{2}}
    \end{pmatrix}
    \in \mathcal{D}_{n-l},
  \end{equation}
  we have $v_{l+1} \gg 1$, see proposition \ref{proposition:Dnproperties}, and so $v_l \gg a_l$ by the hypothesis.
  For $l =n$, we directly have $v_n \gg a_n$ by hypothesis.
  Since also $v_1 \gg v_l$ by lemma \ref{lemma:Grenier}, we have $v_1 v_l \gg a_l^2$, so by taking $a_l$ to be a sufficiently large constant, it follows that $v_1 \geq v_1(\gamma g)$. 

  For $l < n$, if $\bm{c}^{(1)} = 0$ but $
  \begin{pmatrix}
    \bm{c}^{(2)} & \bm{d}^{(2)}
  \end{pmatrix}
  \neq 0$, then we have
  \begin{equation}
    \label{eq:vl1max}
    v_1(\gamma g)^{-1} \geq \bm{c}^{(2)} Y_l \transpose{\bm{c}}^{(2)} + ( \bm{c}^{(2)} X_l + \bm{d}^{(2)}) Y_l^{-1} \transpose{( \bm{c}^{(2)} X_l + \bm{d}^{(2)})} \geq v_{l+1}(g)^{-1}
  \end{equation}
  since $g_l(g) \in \mathcal{D}_{n-l}$.
  We have $v_{l+1}^{-1} \geq a_l v_l^{-1} \gg a_l v_1^{-1}$, so $v_{l+1}^{-1} \geq v_1^{-1}$ for $a_l$ sufficiently large, and it follows that $v_1 \geq v_1(\gamma g)$.

  Now, if $l =n$ or if $\bm{c}^{(1)}$, $\bm{c}^{(2)}$, and $\bm{d}^{(2)}$ are all $0$, then we have $\bm{d}^{(1)} \neq 0$ and
  \begin{equation}
    \label{eq:v1max}
    v_1(\gamma g)^{-1} = \bm{d}^{(1)} \transpose{U}_l^{-1} V_l^{-1} U_l^{-1} \transpose{\bm{d}}^{(1)} \geq v_1^{-1}
  \end{equation}
  as $U_l V_l \transpose{U_l} \in \mathcal{D}_l'$.
  We have verified that for any $\gamma \in \Gamma$, $v_1 \leq v_1(\gamma g)$, which is the first condition defining the fundamental domain $\mathcal{D}_n$.   

  Restricting to $\gamma \in \Gamma_1$, which fixes $v_1(g)$, the same argument as above shows that $v_2(g) \geq v_2(\gamma g)$ for all $\gamma \in \Gamma_1$.
  Continuing this way, we find that the $v_j$, $1\leq j \leq l$ are all maximal (over $\Gamma_{j,2}$), and so, by the construction of $\mathcal{D}_n$, there is a $\gamma \in \Gamma_l$ with the form
  \begin{equation}
    \label{eq:unipotentgamma}
    \gamma =
    \begin{pmatrix}
      A & B \\
      0 & \transpose{A}^{-1}
    \end{pmatrix}
    ,
  \end{equation}
  where $A$ is upper-triangular unipotent (so $\gamma \in \Gamma_l$ for all $l$) such that $\gamma g \in \mathcal{D}_n$.
\end{proof}

\section{Proof of the main theorem}
\label{sec:mainproof}

In the following subsection we gather some technical lemmas regarding the height function needed in the proof of theorem \ref{theorem:thetasumbound2}, see section \ref{sec:mainproof1}.
This height function is motivated by the following corollary from \cite{MarklofWelsh2021a}.

\begin{corollary}
  \label{corollary:upperbound}
  For a Schwartz function $f\in\mathcal{S}(\mathbb{R}^n)$ and $(h, g) \in \tilde{\mathcal{D}}$, and $A > 0$, we have
  \begin{equation}
    \label{eq:thetadetbound}
    \Theta_f(h, g) \ll_{f,A} (\det Y)^{\frac{1}{4}}( 1 + \bm{x} Y \transpose{\!\bm{x}})^{-A}
  \end{equation}
 where
  \begin{equation}
    \label{eq:gXYkQ1}
    g =
    \begin{pmatrix}
      I & X \\
      0 & I 
    \end{pmatrix}
    \begin{pmatrix}
      Y^{\frac{1}{2}} & 0 \\
      0 & \transpose{Y}^{-\frac{1}{2}}
    \end{pmatrix}
    k(Q).
  \end{equation} 
\end{corollary}

We remark that in \cite{MarklofWelsh2021a} this is obtained as a consequence of full asymptotics of the theta function in the various cuspidal regions.
We also remark that in \cite{MarklofWelsh2021a} we use a slightly different fundamental domain, however an examination of the proof there shows that the fundamental domain can be replaced by any set satisfying the conclusions of proposition \ref{proposition:Dnproperties}.

\subsection{Heights and volumes}
\label{sec:heightsvolumes}

For a fixed $A > 0$ sufficiently large depending only on $n$, we define the function $D : \tilde{\Gamma} \backslash (H \rtimes G) \to \mathbb{R}_{> 0}$ by
\begin{equation}
  \label{eq:DAdef}
  D \left( \tilde{\Gamma} (h, g) \right) = \det Y (\gamma g) \left( 1 + \bm{x}(uh_\gamma h^{\gamma^{-1}}) Y(\gamma g) \transpose{ \bm{x}}(u h_\gamma h^{\gamma^{-1}} ) \right)^{-A}
\end{equation}
where $(uh_\gamma, \gamma) \in \tilde{\Gamma}$ is so that $(uh_\gamma,\gamma)(h, g) \in \tilde{\mathcal{D}}$.
Here we write $h \in H$ as $h = ( \bm{x}(h), \bm{y}(h), t(h) )$.
For completeness, in case there are more than one $(uh_\gamma, \gamma) \in \tilde{\Gamma}$ such that $(uh_\gamma, \gamma)(h, g) \in \tilde{\mathcal{D}}$, then we define $D\left( \tilde{\Gamma}(h, g) \right)$ to be the largest of the finite number of values (\ref{eq:DAdef}).
This point is not essential as these values are within constant multiples of each other; see the argument in lemma \ref{lemma:heightcontinuity} for how this can be proved.

We begin by analyzing the growth of the height function.
We let $\tilde{\mu}$ denote the Haar probability measure on $\tilde{\Gamma} \backslash (H \rtimes G)$, which is $\mu$, the Haar probability measure on $\Gamma \backslash G$, times the Lebesgue measure on the entries of $h = (\bm{x}, \bm{y}, t)$.

\begin{lemma}
  \label{lemma:volumecalculation}
  For $R \geq 1$ we have
  \begin{equation}
    \label{eq:secondvolumebound}
    \tilde{\mu} ( \{ \tilde{\Gamma} (h, g) \in \tilde{\Gamma} \backslash (H \rtimes G) : D( \tilde{\Gamma} (h, g)) \geq R  \} ) \ll R^{- \frac{n+2}{2}}
  \end{equation}
  with the implied constant depending only on $n$. 
\end{lemma}

\begin{proof}
  We recall that $g \in \mathcal{D}_n$ is written as
  \begin{equation}
    \label{eq:gUVXQ}
    g =
    \begin{pmatrix}
      U & X\transpose{U}^{-1} \\
      0 & \transpose{U}^{-1} 
    \end{pmatrix}
    \begin{pmatrix}
      V^{\frac{1}{2}} & 0 \\
      0 & V^{-\frac{1}{2}}
    \end{pmatrix}
    k(Q)
  \end{equation}
  for $U$ upper-triangular unipotent, $X$ symmetric, $Q \in \mathrm{U}(n)$, and
  \begin{equation}
    \label{eq:Vrecal}
    V = V(g) =
    \begin{pmatrix}
      v_1 & \cdots & 0 \\
      \vdots & \ddots & \vdots \\
      0 & \cdots & v_n
    \end{pmatrix}
  \end{equation}
  positive diagonal.
  The Haar measure $\mu$ on $G$ is then proportional to Lebesgue measure with respect to the entries of $X$ and the off-diagonal entries of $U$, $\mathrm{U}(n)$-Haar measure on $Q$, and the measure given by
  \begin{equation}
    \label{eq:VHaarmeasure}
    v_1^{-n-1} v_2^{-n} \cdots v_n^{-2} \dd v_1 \dd v_2 \cdots \dd v_n
  \end{equation}
  on $V$.

  By proposition \ref{proposition:Dnproperties}, we observe that the set in (\ref{eq:secondvolumebound}) is contained in the set of $(h, g)$ satisfying $v_j \geq c v_{j+1}$ for all $1\leq j < n$ and some $c > 0$ in addition to $\det Y \geq R$ and $\bm{x} Y \transpose{\bm{x}} \leq R^{-\frac{1}{A}} (\det Y)^{\frac{1}{A}}$.
  Moreover, the variables $\bm{x}, \bm{y}, t$ as well as $U$, $X$ are constrained to compact sets, and so the measure of the set (\ref{eq:secondvolumebound}) is
  \begin{equation}
    \label{eq:secondvolumebound2}
    \ll R^{-\epsilon} \underset{ \substack{ v_j \geq c v_{j+1} \\ v_1 \cdots v_n \geq R}}{\int \cdots \int} v_1^{-n - \frac{3}{2} + \epsilon } v_2^{-n - \frac{1}{2} + \epsilon } \cdots v_n^{-\frac{5}{2} + \epsilon} \dd v_1 \dd v_2 \cdots \dd v_n,
  \end{equation}
  where $\epsilon = \frac{n}{2A}$.

  Changing variables $v_j = \exp(u_j)$, the integral in (\ref{eq:secondvolumebound2}) is
  \begin{multline}
    \label{eq:ujintegral}
    R^{-\epsilon} \underset{ \substack{ u_j - u_{j+1} \geq \log c \\ u_1 + \cdots + u_n \geq \log R}}{\int \cdots \int} \exp \big( -(n + \tfrac{1}{2} - \epsilon)u_1 - (n- \tfrac{1}{2} - \epsilon)u_2 \\
    - \cdots -(\tfrac{3}{2} - \epsilon) u_n \big) \dd u_1 \dd u_2 \cdots \dd u_n. 
  \end{multline}
  We now make the linear change of variables $s_j = u_j - u_{j+1}$ for $ j < n$ and $s_n = u_1 + \cdots + u_n$.
  This transformation has determinant $n$ and its inverse is given by
  \begin{equation}
    \label{eq:ujequals}
    u_j = - \frac{1}{n} \sum_{1\leq i < j} i s_i + \frac{1}{n} \sum_{j\leq i < n} (n-i) s_i + \frac{1}{n} s_n. 
  \end{equation}
  We find that the exponent in (\ref{eq:ujintegral}) is then
  \begin{equation}
    \label{eq:exponent}
    - \sum_{1\leq j \leq n} (n - j + \tfrac{3}{2} - \epsilon ) u_j = - \left(\frac{n+2}{2} - \epsilon\right) s_n - \sum_{1 \leq j < n} \frac{j(n-j)}{2} s_j. 
  \end{equation}
  As $ \frac{j(n-j)}{2} > 0$ for $j < n$, the bound (\ref{eq:secondvolumebound}) follows.
\end{proof}

Lemma  \ref{lemma:heightcontinuity} below contains a key estimate, establishing a kind of `uniform continuity' for $\log D$.
The proof of this lemma is the primary motivation for defining our new fundamental domain and studying the shape of its cusp in sections \ref{sec:fundamentaldomain} and \ref{sec:cuspshape}.
For the proof, we first establish a similar kind of `uniform continuity' for the functions $v_l(\Gamma_l g)$ and $v_{l+1}(\Gamma_l g)$ that are essential to section \ref{sec:cuspshape}. 

\begin{lemma}
  \label{lemma:vlcontinuity}
  Let $g, g_0 \in G$ with $|| g_0 - I|| \leq 1$, then
  \begin{equation}
    \label{eq:vlcontinuity}
    v_l(g) \asymp v_l(g g_0),\ v_l(\Gamma_l g) \asymp v_l(\Gamma_lg g_0),\ v_{l+1}( \Gamma_l g) \asymp v_{l+1} ( \Gamma_l g g_0)
  \end{equation}
  for all $1\leq l \leq n$ with implied constants depending only on $n$.
\end{lemma}

\begin{proof}
  We first note that we may in fact work with $||I - g_0|| \leq \epsilon$ as then the statement would follow by repeated application of the estimates.
  In fact, we may assume $|| I - g_0^{-1} || \leq \epsilon$ as well. 
  Now write
  \begin{equation}
    \label{eq:gIwasawa}
    g =
    \begin{pmatrix}
      I & X \\
      0 & I 
    \end{pmatrix}
    \begin{pmatrix}
      Y^{\frac{1}{2}} & 0 \\
      0 & \transpose{Y}^{-\frac{1}{2}}
    \end{pmatrix}
    \begin{pmatrix}
      R & -S \\
      S & R
    \end{pmatrix}
    ,
  \end{equation}
  with $R + \mathrm{i}S \in \mathrm{U}(n)$, so in particular $R\transpose{R} + S \transpose{S} = I$.
  With $g_0 =
  \begin{pmatrix}
    A & B \\
    C & D
  \end{pmatrix}
  $, we have from (\ref{eq:XYQ}) that
  \begin{multline}
    \label{eq:Ygg1}
    Y(gg_0)^{-1} = \transpose{Y}^{-\frac{1}{2}} \big( S A \transpose{A} \transpose{S} + R C \transpose{A} \transpose{S} + S A \transpose{C} \transpose{R} + R C \transpose{C} R \\
    + S B\transpose{B} \transpose{S} + R D \transpose{B} \transpose{S} + S B \transpose{D} \transpose{R} + R D \transpose{D} \transpose{R}\big) Y^{-\frac{1}{2}}. 
  \end{multline}
  As $|| g_0 - I || \leq \epsilon$, we have
  \begin{equation}
    \label{eq:Ygg1I}
    \transpose{Y}(gg_0)^{-\frac{1}{2}} = \transpose{Y}^{- \frac{1}{2}}( I + O (\epsilon)).
  \end{equation}
  On the other hand, letting $\bm{y}_j$ and $\bm{y}_J'$ denote the rows of $\transpose{Y}^{-\frac{1}{2}}$ and $ \transpose{Y}(g g_0)^{-\frac{1}{2}}$, we have
  \begin{equation}
    \label{eq:v1formula}
    v_1(g)^{-\frac{1}{2}} = || \bm{y}_1 ||,\quad v_1(g g_0)^{-\frac{1}{2}} = || \bm{y}_1'||
  \end{equation}
  and for $2 \leq l \leq n$,
  \begin{equation}
    \label{eq:vlformula}
    v_{l}(g)^{-\frac{1}{2}} = \frac{|| \bm{y}_1 \wedge \cdots \wedge \bm{y}_l ||}{|| \bm{y}_1 \wedge \cdots \wedge \bm{y}_{l-1}||},\quad v_{l}(gg_0)^{-\frac{1}{2}} = \frac{|| \bm{y}_1' \wedge \cdots \wedge \bm{y}_l' ||}{|| \bm{y}_1' \wedge \cdots \wedge \bm{y}_{l-1}'||},
  \end{equation}
  and so $v_{l}(g) \asymp v_l(g g_0)$ follows.
  
  Now let $\gamma \in \Gamma_l$ be so that $v_l(\Gamma_l g) = v_l(\gamma g)$.
  We have
  \begin{equation}
    \label{eq:vlGammalgg1}
    v_l(\Gamma_l g g_0) \leq v_l( \gamma g g_0) \ll v_l(\gamma g) = v_l(\Gamma_l g),
  \end{equation}
  and the reverse bound follows by switching the roles of $g$ and $g g_0$, and using $|| g_0^{-1} - I || \leq \epsilon$.
  The final estimate in (\ref{eq:vlcontinuity}) is proved in the same way. 
\end{proof}

\begin{lemma}
  \label{lemma:heightcontinuity}
  If $(h, g), (h_0, g_0) \in G$ with $|| g_0 - I || \leq 1$ and $h_0 = (\bm{x}_0, \bm{y}_0, t_0)$ satisfies $||\bm{x}_0||, || \bm{y}_0|| \leq 1$, then
  \begin{equation}
    \label{eq:heightcontinuity}
    D(\tilde{\Gamma}(h, g)) \asymp D(\tilde{\Gamma}(h, g)(h_0,g_0)).
  \end{equation}
\end{lemma}

\begin{proof}
  We observe as in lemma \ref{lemma:vlcontinuity}, we may in fact assume
  \begin{equation}
    \label{eq:epsilonbounds}
    ||g_0 - I||\leq \epsilon,\ ||\bm{x}_0|| \leq \epsilon,\ \mathrm{and\ } ||\bm{y}_0 || \leq \epsilon.
  \end{equation}
  Moreover, it suffices to show that $D(\tilde{\Gamma} (h, g)(h_0,g_0)) \gg D(\tilde{\Gamma} (h, g))$ as the other inequality follows from switching $(h, g)$ and $(h, g)(h_0, g_0)$ as we may assume in addition that $(h_0, g_0)^{-1} = (h_0^{-g_0}, g_0^{-1})$ also satisfies (\ref{eq:epsilonbounds}). 

  Now let us suppose that $(h, g) \in \tilde{\mathcal{D}}$ so that
  \begin{equation}
    \label{eq:Dhg}
    D( \tilde{\Gamma}(h, g)) = (\det Y(g))( 1 + \bm{x}(h) Y(g) \transpose{\bm{x}(h)})^{-A}.
  \end{equation}
  Let $1\leq l \leq n$ be the largest index such that $v_l(g) \geq a v_{l+1} (g)$ (or $v_n(g) \geq a$ when $l = n$) where $a$ is a constant determined by the constants in proposition \ref{proposition:vllarge} and lemma \ref{lemma:vlcontinuity}.
  If no such $l$ exists, then we have $v_j(g) \asymp 1$ for all $j$, and lemma \ref{lemma:vlcontinuity} implies that $v_j(g g_0) \asymp 1$ as well.
  The bounds
  \begin{equation}
    \label{eq:compactbound}
    D(\tilde{\Gamma}(h, g)(h_0, g_0)) \gg 1 \gg D(\tilde{\Gamma}(h, g))
  \end{equation}
  then follow immediately.

  Now assuming that such a maximal $l$ exists, we have that $v_j(g) \asymp 1$ for all $j > l$.
  For these $j$, lemma \ref{lemma:vlcontinuity} then implies that $v_j(g g_0) \asymp 1$, and it follows that $v_j(\gamma g g_0) \asymp 1$ for $\gamma \in \Gamma_l$ such that $g_l(\gamma g g_0) \in \mathcal{D}_{n-l}$, see (\ref{eq:gldef}).
  By lemma \ref{lemma:Dminvl}, we have $v_l(\Gamma_l g) \gg v_l(g)$, and so
  \begin{equation}
    \label{eq:vl1g1}
    v_l(\Gamma_l g) \gg a v_{l+1}(g) = a v_{l+1} (\Gamma_l g)
  \end{equation}
  since $g_l(g) \in \mathcal{D}_{n-l}$.
  Via lemma \ref{lemma:vlcontinuity}, this implies that $v_l(\Gamma_l g g_0) \gg a v_{l+1}(\Gamma_l g g_0)$, so $a$ can be chosen large enough so that $g g_0$ satisfies the hypotheses of proposition \ref{proposition:vllarge}, and we let $\gamma \in \Gamma_l$ be so that $\gamma g g_0 \in \mathcal{D}$.

  We write
  \begin{equation}
    \label{eq:A1def}
    \gamma =
    \begin{pmatrix}
      A_1 & * & * & * \\
      0 & * & * & * \\
      0 & 0 & * & 0 \\
      0 & * & * & *
    \end{pmatrix}
    ,
  \end{equation}
  where $A_1 \in \mathrm{GL}(l, \mathbb{Z})$.
  From the estimates above, we have
  \begin{multline}
    \label{eq:detYgammag2}
    \det Y(\gamma g g_0) \asymp \det U_l(\gamma g g_0)V_l(\gamma g g_0)\transpose{U_l}(\gamma g g_0) = \det U_l(g g_0)V_l(g g_0)\transpose{U_l}(g g_0) \\
    \asymp \det U_l(g) V_l(g) \transpose{U}_l(g) \asymp \det Y(g),
  \end{multline}
  where the equality follows from the fact that $\gamma \in \Gamma_l$ normalizes the first matrix in (\ref{eq:Pjdecomp}) and $\det A_1 = \pm 1$.

  It now remains to consider the factors $1 + \bm{x}(*)Y(*) \transpose{\bm{x}(*)}$ in the definition of the height function $D$.
  Let $u = (\bm{m}, \bm{n}, 0)$ with $\bm{m}, \bm{n} \in \mathbb{Z}^n$ be so that $(uh_\gamma, \gamma)(h,g)(h_0, g_0) \in \tilde{\mathcal{D}}$. 
  Recalling the definition of $h_\gamma = (\bm{r}, \bm{s}, 0)$ following (\ref{eq:ABCDdef2}), we have that $\bm{r}^{(1)} = 0$ where $\bm{r} =
  \begin{pmatrix}
    \bm{r}^{(1)} & \bm{r}^{(2)}
  \end{pmatrix}
  $.
  Moreover, writing $\bm{x} =
  \begin{pmatrix}
    \bm{x}^{(1)} & \bm{x}^{(2)}
  \end{pmatrix}
  $, we have $\bm{x}^{(1)}( (h h_0^{g^{-1}})^{\gamma^{-1}}) = \bm{x}^{(1)}(h h_0^{g^{-1}}) A_1^{-1}$.
  Using proposition \ref{proposition:Dnproperties} together with the fact that $u$ minimizes the absolute values of the entries of $\bm{x}(u h_\gamma (h h_0^{g^{-1}})^{\gamma^{-1}})$, we have
  \begin{multline}
    \label{eq:xUVUx}
    1 + \bm{x}( u h_\gamma (h h_0^{g^{-1}})^{\gamma^{-1}}) Y(\gamma g g_0) \transpose{\bm{x}}(u h_\gamma (h h_0^{g^{-1}})^{\gamma^{-1}}) \\
    \ll 1 + \bm{x}( h_\gamma (h h_0^{g^{-1}})^{\gamma^{-1}}) Y(\gamma g g_0) \transpose{\bm{x}}(h_\gamma (h h_0^{g^{-1}})^{\gamma^{-1}}),
  \end{multline}
  and from the estimates above on the $v_j(\gamma g g_0)$ for $j > l$, we have
  \begin{multline}
    \label{eq:xUVUx1}
    1 + \bm{x}( h_\gamma (h h_0^{g^{-1}})^{\gamma^{-1}}) Y(\gamma g g_0) \transpose{\bm{x}}(h_\gamma (h h_0^{g^{-1}})^{\gamma^{-1}}) \\
    \asymp 1 + \bm{x}^{(1)}( h_\gamma (h h_0^{g^{-1}})^{\gamma^{-1}}) U_l(\gamma g g_0) V_l(\gamma g g_0) \transpose{U_l}(\gamma g g_0)\transpose{\bm{x}}(h_\gamma (h h_0^{g^{-1}})^{\gamma^{-1}}).
  \end{multline}
  Using the expressions for $h_\gamma$, $(h h_0^{g^{-1}})^{\gamma^{-1}}$, and that
  \begin{equation}
    \label{eq:AUVUA}
    U_l(\gamma g g_0) V_l(\gamma g g_0) \transpose{U_l}(\gamma g g_0) = A_1 U_l(g g_0) V_l(g g_0) \transpose{U_l}(g g_0)\transpose{A}_1,
  \end{equation}
  the right side of (\ref{eq:xUVUx1}) is equal to
  \begin{equation}
    \label{eq:xUVUxg2}
    1 + \bm{x}^{(1)}(h h_0^{g^{-1}}) U_l( g g_0) V_l(g g_0) \transpose{U}_l(g g_0) \transpose{\bm{x}}^{(1)}( h h_0^{g^{-1}}) \asymp 1 + \bm{x}(h h_0^{g^{-1}}) Y(g g_0) \transpose{\bm{x}}(h h_0^{g^{-1}})
  \end{equation}
  by the above bounds on $v_j(g g_0)$ for $j > l$. 

  Recalling that
  \begin{equation}
    \label{eq:g1Iwasawa}
    g =
    \begin{pmatrix}
      I & X(g) \\
      0 & I
    \end{pmatrix}
    \begin{pmatrix}
      Y(g)^{\frac{1}{2}} & 0 \\
      0 & \transpose{Y}(g)^{-\frac{1}{2}}
    \end{pmatrix}
    k(g)
  \end{equation}
  with $k(g) \in K = G \cap \mathrm{SO}(2n, \mathbb{R})$, we set $h_0' = h_0^{k(g)^{-1}}$ and note that
  \begin{equation}
    \label{eq:korthogonal}
    ||\bm{x}(h_0')||^2 + ||\bm{y}(h_0')||^2 = ||\bm{x}(h_0)||^2 + ||\bm{y}(h_0)||^2.
  \end{equation}
  Since $Y( g g_0) = Y(g)^{\frac{1}{2}} Y(k(g)g_0) \transpose{Y}(g)^{\frac{1}{2}}$ and $\bm{x}(h h_0^{g^{-1}}) = \bm{x}(h) + \bm{x}(h_0') Y(g)^{-\frac{1}{2}}$, the right side of (\ref{eq:xUVUxg2}) is equal to
  \begin{multline}
    \label{eq:g1xYx}
    1 + \bm{x}(h) Y(g)^{\frac{1}{2}} Y(k(g) g_0) \transpose{Y}(g)^{-\frac{1}{2}} \transpose{\bm{x}}(h) \\
    + 2 \bm{x}(h) Y(g)^{\frac{1}{2}}Y(k(g) g_0)\transpose{\bm{x}}(h_0') + \bm{x}(h_0') Y(k(g)g_0) \transpose{\bm{x}}(h_0'). 
  \end{multline}
  We have that $|| g_0 - I || \leq \epsilon$ implies $Y(k(g) g_0) = I + O(\epsilon)$ as in (\ref{eq:Ygg1}), so if (\ref{eq:korthogonal}) is at most $\epsilon^2$ as well, with $\epsilon$ sufficiently small, then (\ref{eq:g1xYx}) is
  \begin{equation}
    \label{eq:g1xYxbound}
    \asymp 1 + \bm{x}(h) Y(g) \transpose{\bm{x}}(h),
  \end{equation}
  where we have used
  \begin{multline}
    \label{eq:cauchy}
    2| \bm{x}(h) Y(g)^{\frac{1}{2}}Y(k(g) g_0)\transpose{\bm{x}}(h_0') | \\
    \leq \sqrt{ \bm{x}(h_0') Y(k(g) g_0)^2 \transpose{\bm{x}}(h_0')} \left(\bm{x}(h) Y (g)\transpose{\bm{x}}(h) + 1\right) \ll \epsilon \left(\bm{x}(h) Y (g)\transpose{\bm{x}}(h) + 1\right)
  \end{multline}
  to bound the third term in (\ref{eq:g1xYx}).
  The bound $D(\tilde{\Gamma}(h, g)(h_0, g_0)) \gg D( \tilde{\Gamma}(h, g)$ now follows.
\end{proof}

\subsection{Proof of theorem \ref{theorem:thetasumbound2}}
\label{sec:mainproof1}

We recall the following lemma from \cite{MarklofWelsh2021a}.

\begin{lemma}
  \label{lemma:cutoff}
  There exists a smooth, compactly supported function $f_1 :  \mathbb{R} \to \mathbb{R}_{\geq 0}$ such that
  \begin{equation}
    \label{eq:chidecomp}
    \chi_1(x) = \sum_{j \geq 0} \left( f_1\left( 2^j x\right) +  f_1\left( 2^j( 1 -x)\right)\right),
  \end{equation}
  where $\chi_1$ is the indicator function of the open unit interval $(0,1)$. 
\end{lemma}

Now, following the method of \cite{MarklofWelsh2021a}, we define for a subset $S \subset \{1, \dots, n\}$ and $\bm{j} = (j_1, \dots, j_n) \in \mathbb{Z}^n$ with $j_i \geq 0$,
\begin{equation}
  \label{eq:gjSdef}
  g_{\bm{j}, S} =
  \begin{pmatrix}
    A_{\bm{j}} E_S & 0 \\
    0 & A_{\bm{j}}^{-1} E_S
  \end{pmatrix}
  \in G
\end{equation}
where $E_S$ is diagonal with $(i,i)$ entry $-1$ if $i \in S$, $+1$ if $i \not\in S$, and
\begin{equation}
  \label{eq:Ajdef}
  A_{\bm{j}} =
  \begin{pmatrix}
    2^{j_1} & \cdots & 0 \\
    \vdots & \ddots & \vdots \\
    0 & \cdots & 2^{j_n}
  \end{pmatrix}
  . 
\end{equation}
We also set $h_S = (\bm{x}_S, 0, 0) \in H$ where $\bm{x}_S$ has $i$th entry $-1$ if $i \in S$ and $0$ if $i \not\in S$.

As in \cite{MarklofWelsh2021a}, we have
\begin{equation}
  \label{eq:chindecomp}
  \chi_{\mathcal{B}}(\bm{x}) = \sum_{\bm{j}\geq 0} \sum_{S \subset \{1, \dots, n\}} f_n\left( (\bm{x}B^{-1} + \bm{x}_S)A_{\bm{j}} E_S \right),
\end{equation}
where $\chi_{\mathcal{B}}$ is the indicator function of the rectangular box $\mathcal{B} = (0, b_1) \times \cdots \times (0, b_n)$, $B$ is the diagonal matrix with entries $b_1, \dots, b_n$, 
\begin{equation}
  \label{eq:fndef}
  f_n(x_1, \dots, x_n) = \prod_{1\leq j \leq n} f_1(x_j),
\end{equation}
and the sums are over $\bm{j} \in \mathbb{Z}^n$ with nonnegative entries.

Let $\psi: [0,\infty) \to [1, \infty)$ be an increasing function.
Then for $C > 0$ we define $\mathcal{G}_{\bm{j}}(\psi, C)$ to be the set of $\tilde{\Gamma} (h, g) \in \tilde{\Gamma} \backslash (H \rtimes G)$ such that
\begin{equation}
  \label{eq:GjpsiCdef}
  D\big( \tilde{\Gamma} (h, g) ( 1, 
  \begin{pmatrix}
    \e^{-s}I & 0 \\
    0 & \e^s I
  \end{pmatrix}
  )
  (h_S, g_{\bm{j}, S}) \big)^{\frac{1}{4}} \leq C \psi(s)
\end{equation}
for all $S \subset \{1, \dots, n\}$ and $s\geq 1$. 

\begin{lemma}
  \label{lemma:psivolumebound}
  Suppose that $\psi$ satisfies
  \begin{equation}
    \label{eq:psicondition}
    \int_0^\infty \psi(x)^{-(2n + 4)} \dd x \leq C_\psi
  \end{equation}
  for some $C_\psi \geq 1$. Then
  \begin{equation}
    \label{eq:Gjvolumebound}
    \tilde{\mu} \left( \tilde{\Gamma} \backslash (H \rtimes G) - \mathcal{G}_{\bm{j}} (\psi, C) \right) \ll C_\psi C^{-( 2n + 4)} 2^{j_1 + \cdots + j_n}.
  \end{equation}
\end{lemma}

\begin{proof}
  Suppose that $\tilde{\Gamma}(h, g) \not\in \mathcal{G}_{\bm{j}}( \psi, C)$, so there exists $S \subset \{1,\dots, n\}$ and $s\geq 1$ such that
  \begin{equation}
    \label{eq:DAlarge}
    D \big( \tilde{\Gamma}( h, g) (1,
    \begin{pmatrix}
      \e^{-s} I & 0 \\
      0 & \e^s I 
    \end{pmatrix}
    ) ( h_S, g_{\bm{j}, S}) \big)^{\frac{1}{4}} \geq C\psi(s).
  \end{equation}
  We let $k$ be a nonnegative integer such that
  \begin{equation}
    \label{eq:kdef}
    \frac{k}{K_{\bm{j}}} \leq s < \frac{k + 1}{K_{\bm{j}}},
  \end{equation}
  where $K_{\bm{j}} = K 2^{j_1 + \cdots + j_n}$ with $K$ a constant to be determined.
  We have
  \begin{equation}
    \label{eq:h1g1def}
    (1,
    \begin{pmatrix}
      \e^{-s} I & 0 \\
      0 & \e^s I 
    \end{pmatrix}
    ) (h_S, g_{\bm{j}, S}) = (1,
    \begin{pmatrix}
      \e^{-\frac{k}{K_{\bm{j}}}} I & 0 \\
      0 & \e^{\frac{k}{K_{\bm{j}}}} I 
    \end{pmatrix}
    )(h_S, g_{\bm{j}, S}) (h_1, g_1),
  \end{equation}
  where, with $s' = s - \frac{k}{K_{\bm{j}}}$,
  \begin{equation}
    \label{eq:explicith1g1}
    h_1 = ( (\e^{s'} -1) \bm{x}_S A_{\bm{j}} E_S, 0, 0),\quad g_1 =
    \begin{pmatrix}
      \e^{-s'} I & 0 \\
      0 & \e^{s'} I 
    \end{pmatrix}
    .
  \end{equation}
  As $|s'| \leq K_{\bm{j}}^{-1}$, we can make $K$ sufficiently large so that $(h_1, g_1)$ satisfies the conditions of lemma \ref{lemma:heightcontinuity}.
  From this and the fact that $\psi$ is increasing, we have that
  \begin{equation}
    \label{eq:DAlargeinteger}
    D \big( \tilde{\Gamma}( h, g) (1,
    \begin{pmatrix}
      \e^{-\frac{k}{K_{\bm{j}}}} I & 0 \\
      0 & \e^{\frac{k}{K_{\bm{j}}}} I 
    \end{pmatrix}
    ) ( h_S, g_{\bm{j}, S}) \big)^{\frac{1}{4}} \gg C \psi\left( \frac{k}{K_{\bm{j}}} \right).
  \end{equation}
  By lemma \ref{lemma:volumecalculation} and the fact that right multiplication is volume preserving, we have that the set of $\tilde{\Gamma}(h, g)$ satisfying (\ref{eq:DAlargeinteger}) has $\tilde{\mu}$-volume bounded by a constant times
  \begin{equation}
    \label{eq:DAlargevolumebound}
    C^{-2n -4} \psi\left(\frac{k}{K_{\bm{j}}} \right)^{-2n -4}.
  \end{equation}
  Bounding the volume of the set $\tilde{\Gamma} \backslash (H \rtimes G) - \mathcal{G}_{\bm{j}}( \psi, C)$ by summing (\ref{eq:DAlargevolumebound}) over $S\subset \{1, \dots, n\}$ and nonnegative $k \in \mathbb{Z}$, we obtain the bound
  \begin{equation}
    \label{eq:ksumvolumebound}
    C^{-(2n + 4)} \sum_{k\geq 0} \psi\left( \frac{k}{K_{\bm{j}}} \right)^{-(2n + 4)} \ll C^{-(2n + 4)}\left( \psi(0) +  \int_0^\infty \psi\left( \frac{x}{K_{\bm{j}}} \right)^{-(2n + 4)} \dd x\right) 
  \end{equation}
  as $\psi(x) $ is increasing.
  The bound (\ref{eq:Gjvolumebound}) follows by changing variables.   
\end{proof}

We now proceed to the proof of theorem \ref{theorem:thetasumbound2}.

\begin{proof}[Proof of theorem \ref{theorem:thetasumbound2}]
    From (\ref{eq:chindecomp}) we express $\theta_{\mathcal{B}}(M, X, \bm{x}, \bm{y})$ as 
  \begin{multline}
    \label{eq:Sdecomp}
    \sum_{S \subset \{1, \dots, n\}} \sum_{\bm{j} \geq 0} \sum_{\bm{m}\in \mathbb{Z}^n}f_n \left( \frac{1}{M}(\bm{m} + \bm{x} + M\bm{x}_S B) B^{-1} E_S A_{\bm{j}}\right) \e\left( \frac{1}{2}\bm{m} X \transpose{\!\bm{m}} + \bm{m} \transpose{\!\bm{y}} \right).
  \end{multline}
  We break the sum in (\ref{eq:Sdecomp}) into terms $\bm{j}$ such that $2^{j_i} b_{j_i}^{-1} \leq M$ for all $i$ and terms $\bm{j}$ such that $2^{j_i} b_{j_i}^{-1} > M$ for some $i$.
  Using (\ref{eq:Thetaexample}), we write the first part as
  \begin{equation}
    \label{eq:smallj}
    \e( \tfrac{1}{2} \bm{x}X\transpose{\bm{x}}) M^{\frac{n}{2}} (\det B)^{\frac{1}{2}} \sum_{\substack{ \bm{j} \geq 0 \\ 2^{j_i} b_{j_i}^{-1} \leq M}} 2^{-\frac{1}{2} ( j_1 + \cdots + j_n)} \Theta_{f_n}\left( (h, g(MB, X))(h_S, g_{\bm{j}, S}) \right),
  \end{equation}
  where $h = (\bm{x}, \bm{y} - \bm{x}X , 0)$ and
  \begin{equation}
    \label{eq:gMBXdef}
    g(MB, X) =
    \begin{pmatrix}
      I & X \\
      0 & I
    \end{pmatrix}
    \begin{pmatrix}
      \frac{1}{M} B^{-1} & 0 \\
      0 & M B
    \end{pmatrix}
    .
  \end{equation}
  Bounding this is the main work of the proof, but we first bound the contribution of the terms $\bm{j}$ with a large index.

  Suppose that $L \subset \{1, \dots, n\}$ is not empty and that $2^{j_l} > b_{j_l} M$ for all $l \in L$.
  Then the compact support of $f_1$ implies that the sum over $\bm{m}^{(L)}$, the vector of entries of $\bm{m}$ with index in $L$, has a bounded number of terms.
  We write
  \begin{equation}
    \label{eq:miform}
    \bm{m} X \transpose{\!\bm{m}} = \bm{m}^{(L)} X^{(L,L)} \transpose{\!\bm{m}^{(L)} } + 2 \bm{m}^{(L)} X^{(L,L')} \transpose{\!\bm{m}^{(L')} } +  \bm{m}^{(L')}  X^{(L',L')} \transpose{\!\bm{m}^{(L')}},
  \end{equation}
  where $L'$ is the complement of $L$, and $X^{(L_1,L_2)}$ is the matrix of entries of $X$ with row and column indices in $L_1$ and $L_2$ respectively.
  We have (\ref{eq:fndef}) that $ f_n\left( \frac{1}{M} (\bm{m} + \bm{x} + M \bm{x}_S B) B^{-1} E_S A_{\bm{j}} \right) $ factors as
  \begin{multline}
    \label{eq:fLfLprime}
    f_{\# L} \left( \frac{1}{M} ( \bm{m}^{(L)} + \bm{x}^{(L)} + M\bm{x}_S^{(L)} ) (B^{(L,L)})^{-1} E_S^{(L,L)} A_{\bm{j}}^{(L,L)} \right) \\
    \times f_{\# L'} \left( \frac{1}{M}( \bm{m}^{(L')} + \bm{x}^{(L')} + M\bm{x}_S^{(L')} ) (B^{(L',L')})^{-1} E_S^{(L',L')} A_{\bm{j}}^{(L',L')} \right),
  \end{multline}
  and so, by inclusion-exclusion and the boundedness of $f_{\#L}$, the terms $\bm{j}$ of (\ref{eq:Sdecomp}) with $\bm{j}_l > b_{j_i} M$ for some $i$ is at most a constant times
  \begin{equation}
    \label{eq:jlarge}
    \sum_{\substack{L \subset \{1, \dots, n\} \\ L \neq \emptyset}} \sum_{S \subset L} \sum_{\bm{m}^{(L)}} \big|\theta_{\mathcal{B}^{(L')}}(M, X^{L',L'}, \bm{x}^{(L')}, \bm{y}^{(L')} + \bm{m}^{(L)} X^{(L,L')})\big|,
  \end{equation}
  where the sum over $\bm{m}^{(L)}$ has a bounded number of terms, $\mathcal{B}^{(L')}$ is the edge of $\mathcal{B}$ associated to $L'$, and we have used the decomposition (\ref{eq:chindecomp}) to express $\theta_{\mathcal{B}^{(L')}}(M, X^{L',L'}, \bm{x}^{(L')}, \bm{y}^{(L')} + \bm{m}^{(L)} X^{(L,L')})$ as
  \begin{multline}
    \label{eq:reversedecomp}
    \sum_{S' \subset L'} \sum_{\bm{j}_{L'}} \sum_{\bm{m}_{L'}} f_{\# L'} \left( \frac{1}{M}( \bm{m}^{(L')} + \bm{x}^{(L')} + M\bm{x}_S^{(L')} ) (B^{(L',L')})^{-1} E_S^{(L',L')} A_{\bm{j}}^{(L',L')} \right) \\
    \times \e\left( \tfrac{1}{2}\bm{m}^{(L')} X^{(L',L')} \transpose{\!\bm{m}^{(L')} } + \bm{m}^{(L')} \transpose{(\bm{y}^{(L')} + \bm{m}^{(L)} X^{(L,L')})} \right).
  \end{multline}

  When $L = \{1, \dots, n\}$, the corresponding part of (\ref{eq:jlarge}) is clearly bounded.
  For any other $L$, we may apply theorem \ref{theorem:thetasumbound1} (emphasizing the importance of the uniformity in $\bm{y}$) to conclude for any $\epsilon > 0$, there are full measure sets $\mathcal{X}^{(n- \#L)} = \mathcal{X}^{(n- \#L)}(\epsilon)$ such that if $X^{(L',L')} \in \mathcal{X}^{(n- \#L)}$, the corresponding part of (\ref{eq:jlarge}) is $\ll M^{ \frac{n - \#L}{2} + \epsilon}$ for any $\epsilon > 0$.
  It follows that (\ref{eq:jlarge}) is $\ll M^{\frac{n}{2}}$ assuming that $X$ is such that $X^{(L',L')} \in \mathcal{X}^{(n - \#L)}$ for all nonempty $L \subset \{1, \dots, n\}$.

  We now return to (\ref{eq:smallj}).
  We let $\mathcal{X}_{\bm{j}}(\psi, C)$ to be the set of $(X, \bm{y})  $ with all entries in the interval $(-\tfrac{1}{2}, \tfrac{1}{2}]$ such that there exist $\bm{u}\in (-\tfrac{1}{2}, \tfrac{1}{2})^n$, $A \in \mathrm{GL}(n, \mathbb{R})$ and $T \in \mathbb{R}^{n\times n}_{\mathrm{sym}}$ satisfying
  \begin{equation}
    \label{eq:BAsmall}
    \sup_{B \in \mathcal{K}} || \left((B A)^{-1} - I\right)A_{\bm{j}} || \leq \epsilon,
  \end{equation}
   $|| T || \leq \epsilon$, and
  \begin{equation}
    \label{eq:Xperturb}
    \tilde{\Gamma}\bigg( (\bm{u}, \bm{y} - \bm{u}X, 0),
    \begin{pmatrix}
      I & X \\
      0 & I
    \end{pmatrix}
    \begin{pmatrix}
      A & 0 \\
      0 & \transpose{A}^{-1}
    \end{pmatrix}
    \begin{pmatrix}
      I & 0 \\
      T & I
    \end{pmatrix}
    \bigg) (h_S, g_{\bm{j}, S}) 
    \in \mathcal{G}_{\bm{j}}(\psi, C).
  \end{equation}
  Here we let $\epsilon > 0$ be a sufficiently small constant,
  $\mathcal{G}_{\bm{j}}(\psi, C)$ is defined in (\ref{eq:GjpsiCdef}), and $\mathcal{K}$ is the compact subset from the statement of theorem \ref{theorem:thetasumbound2} identified with the compact subset of positive diagonal matrices $B$ in the obvious way.
  We then set $\mathcal{X}(\psi)$ to be the set of $(X, \bm{y}) \in \mathbb{R}^{n\times n}_{\mathrm{sym}} \times \mathbb{R}^n$ such that
  \begin{multline}
    \label{eq:Xpsidef}
    ( X + R , \bm{y} R + \bm{s}_R + \bm{s}) \in \bigcup_{C > 0} \bigcap \mathcal{X}_{\bm{j}}( \psi, C 2^{a( j_1 + \cdots + j_n)}) \\
    \cap \bigcap_{\substack{ L \subset \{1, \dots, n\} \\ L \neq \emptyset}} \{ (X_1, \bm{y}_1) \in \mathbb{R}^{n\times n} \times \mathbb{R}^n : X_1^{(L',L')} \in \mathcal{X}^{(n - \#L)} \}
  \end{multline}
  for some $(R, \bm{s}) \in \mathbb{Z}^{n\times n} \times \mathbb{Z}^n$, where $\bm{s}_R \in \mathbb{R}^{n}$ has entries $0$ or $\tfrac{1}{2}$ depending on whether the corresponding diagonal entry of $R$ is even or odd, and $a> 0$ is a constant to be determined. 
  
  We first verify that $\mathcal{X}(\psi)$ has full measure, noting that it is enough to show that
  \begin{equation}
    \label{eq:settobound}
    \bigcup_{C > 0} \bigcap_{\bm{j} \geq 0} \mathcal{X}_{\bm{j}}(\psi, C 2^{a( j_1 + \cdots + j_n)})
  \end{equation}
  has full measure in the subset $\mathcal{X}_0$ of $\mathbb{R}^{n\times n}_{\mathrm{sym}} \times \mathbb{R}^n$ having all entries in the interval $(-\tfrac{1}{2}, \tfrac{1}{2}]$.
  Let us suppose that the Lebesgue measure of the complement of $\mathcal{X}_{\bm{j}}(\psi, C)$ in $\mathcal{X}_0$ is greater than some $\delta > 0$, which we assume is small.
  Now, with respect to the measure $( \det A)^{-2n - 1} \prod_{i,j} \dd a_{ij}$ on $\mathrm{GL}(n, \mathbb{R})$, the volume of the set of $A \in \mathrm{GL}(n, \mathbb{R})$ satisfying (\ref{eq:BAsmall}) is within a constant multiple (depending on $\mathcal{K}$) of $2^{-n(j_1 + \cdots + j_n)}$. 
  Then, using the expression (\ref{eq:Haarmeasure1}), (\ref{eq:Haarmeasure2}) for the Haar measure on $H \rtimes G$, we have
  \begin{equation}
    \label{eq:complementmeasure}
    \tilde{\mu}\left( \tilde{\Gamma} \backslash (H \rtimes G) - \mathcal{G}_{\bm{j}}(\psi, C) \right) \gg \delta 2^{-n(j_1 + \cdots + j_n)},
  \end{equation}
  with implied constant depending on $\mathcal{K}$.
  From lemma \ref{lemma:psivolumebound} it follows that
  \begin{equation}
    \label{eq:complementmeasurebound}
    \mathrm{meas} \left( \mathcal{X}_0 - \mathcal{X}_{\bm{j}}(\psi, C) \right) \ll C_{\psi} C^{-2n - 4} 2 ^{(n+1)(j_1 + \cdots + j_n)},
  \end{equation}
  and we find that
  \begin{multline}
    \label{eq:complementmeasurebound1}
    \mathrm{meas}\left( \mathcal{X}_0 - \bigcup_{C > 0} \bigcap_{\bm{j} \geq 0} \mathcal{X}_{\bm{j}}(\psi, C 2^{a(j_1 + \cdots + j_n)}) \right) \\
    \ll  \lim_{C \to \infty} C_\psi C^{-2n- 4} \sum_{\bm{j} \geq 0}  2^{((n+1) - a(2n+4))( j_1 + \cdots + j_n)} = 0
  \end{multline}
  as long as $a > \frac{n+1}{2n + 4}$.

  Now let us suppose that $(X, \bm{y}) \in \mathcal{X}(\psi)$.
  By theorem \ref{theorem:thetaautomorphy}, the size of the theta functions in (\ref{eq:smallj}) is invariant under the transformation on the left of (\ref{eq:Xpsidef}), so we may assume that $X \in \mathcal{X}_0$ as well.
  In particular, we have that $(X, \bm{y})$ is in $\mathcal{X}_{\bm{j}}(\psi, C 2^{a(j_1 + \cdots + j_n)})$ for some $C > 0$ (independent of $\bm{j}$) and all $\bm{j}\geq 0$.
  We have from corollary \ref{corollary:upperbound} and the definition of the height function $D$ that 
  \begin{equation}
    \label{eq:thetabound}
    \ll M^{\frac{n}{2}} \sum_{S \subset \{1, \dots, n\}}  \sum_{\substack{\bm{j} \geq 0 \\ 2^{j_i} b_{j_i}^{-1} \leq  M}}  2^{-\frac{1}{2}( j_1 + \cdots + j_n)} D\left( \tilde{\Gamma}(h, g(MB, X))(h_S, g_{\bm{j},S})\right)^{\frac{1}{4}}
  \end{equation}
  bounds (\ref{eq:smallj}).
  Now for all $\bm{j} \geq 0$ there is a $\tilde{\Gamma}(h',g) \in \mathcal{G}_{\bm{j}}(\psi, C 2^{a(j_1 + \cdots + j_n)})$ with $g$ of the form
  \begin{equation}
    \label{eq:gXAT}
    g =
    \begin{pmatrix}
      I & X \\
      0 & I
    \end{pmatrix}
    \begin{pmatrix}
      A & 0 \\
      0 & \transpose{\!A}^{-1} 
    \end{pmatrix}
    \begin{pmatrix}
      I & 0 \\
      T & I
    \end{pmatrix}
  \end{equation}
  satisfying (\ref{eq:BAsmall}) and $|| T || \leq \epsilon$ and $h'$ having the for $(\bm{u}, \bm{y} - \bm{u}X, 0)$ for some $\bm{u} \in (-\tfrac{1}{2}, \tfrac{1}{2})^n$.
  We have
  \begin{equation}
    \label{eq:grearrange}
    (h', g)( 1, 
    \begin{pmatrix}
      \frac{1}{M} I  & 0 \\
      0 & M I 
    \end{pmatrix}
    ) (h_S, g_{\bm{j}, S})
    = (h, g(MB, X))(h_S,  g_{\bm{j}, S}) (h_1, g_1),
  \end{equation}
  where
  \begin{equation}
    \label{eq:h1def}
    h_1 = \left( - \bm{x}_SA_{\bm{j}} E_S + \bm{x}_S (BA)^{-1} A_{\bm{j}}E_S + \frac{1}{M}( \bm{u} - \bm{x}) B^{-1} A_{\bm{j}} E_{\bm{j}}, 0, 0 \right)
  \end{equation}
  and
  \begin{equation}
    \label{eq:g1def}
    g_1 = g_{\bm{j},S}^{-1}
    \begin{pmatrix}
      BA & 0 \\
      0 & \transpose{(BA)}^{-1}
    \end{pmatrix}
    \begin{pmatrix}
      I & 0 \\
      \frac{1}{M^2} T & I
    \end{pmatrix}
    g_{\bm{j}, S}. 
  \end{equation}
  Recalling that $2^{j_i} \leq M$, the conditions (\ref{eq:BAsmall}) and $|| T || \leq \epsilon$ implies that $(h_1, g_1)$ satisfies the conditions of lemma \ref{lemma:heightcontinuity} for all $M$, which then implies
  \begin{multline}
    \label{eq:DgDgMX}
    D(\tilde{\Gamma} (h, g(MB, X) (h_S, g_{\bm{j}, S}))^{\frac{1}{4}} \asymp D\left(  \tilde{\Gamma}(h',g)(1,
      \begin{pmatrix}
        \frac{1}{M} I & 0 \\
        0 & M I
      \end{pmatrix}
      ) (h_S, g_{\bm{j}, S}) \right)^{\frac{1}{4}} \\
    \ll C 2^{a(j_1 + \cdots + j_n)} \psi(\log M)
  \end{multline}
  since $(h',g) \in \mathcal{G}_{\bm{j}}(\psi, C 2^{a(j_1 + \cdots + j_n)})$.
  Taking $a = \frac{2n + 3}{4n + 8}$  so that $\frac{n+1}{2n + 4} < a < \frac{1}{2}$, it follows that (\ref{eq:thetabound}) is bounded by
  \begin{equation}
    \label{eq:thetabound1}
    \ll C M^{\frac{n}{2}}\psi(\log M) \sum_{\bm{j} \geq 0} 2^{- (\frac{1}{2} -a)(j_1 + \cdots + j_n)}  \ll C M^{\frac{n}{2}} \psi( \log M),
  \end{equation}
  and theorem \ref{theorem:thetasumbound2} follows.
\end{proof}

\bibliographystyle{plain}
\bibliography{references.bib}

\begin{thebibliography}{10}

\bibitem{Borel1969}
Armand Borel.
\newblock {\em Introduction aux groupes arithm\'{e}tiques}.
\newblock Publications de l'Institut de Math\'{e}matique de l'Universit\'{e} de
  Strasbourg, XV. Actualit\'{e}s Scientifiques et Industrielles, No. 1341.
  Hermann, Paris, 1969.

\bibitem{BorelJi2006}
Armand Borel and Lizhen Ji.
\newblock Compactifications of locally symmetric spaces.
\newblock {\em J. Differential Geom.}, 73(2):263--317, 2006.

\bibitem{ConsentinoFlaminio2015}
Salvatore Cosentino and Livio Flaminio.
\newblock Equidistribution for higher-rank {A}belian actions on {H}eisenberg
  nilmanifolds.
\newblock {\em J. Mod. Dyn.}, 9:305--353, 2015.

\bibitem{FedotovKlopp2012}
Alexander Fedotov and Fr\'{e}d\'{e}ric Klopp.
\newblock An exact renormalization formula for {G}aussian exponential sums and
  applications.
\newblock {\em Amer. J. Math.}, 134(3):711--748, 2012.

\bibitem{FiedlerJurkatKorner1977}
H.~Fiedler, W.~Jurkat, and O.~K\"{o}rner.
\newblock Asymptotic expansions of finite theta series.
\newblock {\em Acta Arith.}, 32(2):129--146, 1977.

\bibitem{Grenier1988}
Douglas Grenier.
\newblock Fundamental domains for the general linear group.
\newblock {\em Pacific J. Math.}, 132(2):293--317, 1988.

\bibitem{Knapp2002}
Anthony~W. Knapp.
\newblock {\em Lie groups beyond an introduction}, volume 140 of {\em Progress
  in Mathematics}.
\newblock Birkh\"{a}user Boston, Inc., Boston, MA, second edition, 2002.

\bibitem{LionVergne1980}
G\'{e}rard Lion and Mich\`ele Vergne.
\newblock {\em The {W}eil representation, {M}aslov index and theta series},
  volume~6 of {\em Progress in Mathematics}.
\newblock Birkh\"{a}user, Boston, Mass., 1980.

\bibitem{MarklofWelsh2021c}
Jens Marklof and Matthew Welsh.
\newblock {S}egal-{S}hale-{W}eil representation, theta functions, and
  applications.
\newblock In preparation, 2022.

\bibitem{MarklofWelsh2021a}
Jens Marklof and Matthew Welsh.
\newblock Bounds for theta sums in higher rank {I}.
\newblock {\em J. d'Analyse Math.}, 2023.

\bibitem{Mumford1983}
David Mumford.
\newblock {\em Tata lectures on theta. {I}}, volume~28 of {\em Progress in
  Mathematics}.
\newblock Birkh\"{a}user Boston, Inc., Boston, MA, 1983.
\newblock With the assistance of C. Musili, M. Nori, E. Previato and M.
  Stillman.

\bibitem{Terras1988}
Audrey Terras.
\newblock {\em Harmonic analysis on symmetric spaces and applications. {II}}.
\newblock Springer-Verlag, Berlin, 1988.

\end{thebibliography}

\noindent
JM: School of Mathematics, University of Bristol, Bristol BS8 1UG, U.K.\\

\noindent
MW: Department of Mathematics, University of Maryland, College Park, MD 20742, USA

\end{document}